\documentclass[12pt, reqno]{amsart}
\usepackage{amsfonts,amssymb,amscd,amsmath,enumerate,verbatim,calc,graphicx}
\usepackage{enumitem}
\usepackage{amsmath}
\usepackage{amssymb}
\usepackage{amscd}
\usepackage{wrapfig}
\usepackage{float}
\usepackage{color}
\usepackage{mathtools}
\usepackage{tikz-cd}
\usepackage{blkarray}
\usepackage[utf8]{inputenc}
\usepackage[colorlinks,pagebackref=true]{hyperref}
\makeatletter

\makeatletter
\def\widebreve#1{\mathop{\vbox{\m@th\ialign{##\crcr\noalign{\kern3\p@}%
				\brevefill\crcr\noalign{\kern3\p@\nointerlineskip}%
				$\hfil\displaystyle{#1}\hfil$\crcr}}}\limits}

\def\brevefill{$\m@th \setbox\z@\hbox{$\braceld$}%
	\bracelu\leaders\vrule \@height\ht\z@ \@depth\z@\hfill\braceru$}
\makeatletter

\def\@citecolor{blue}
\def\@linkcolor{blue}
\def\@urlcolor{blue}
\def\@urlcolor{blue}

\def\NZQ{\mathbb}               

\def\ZZ{{\NZQ Z}}

\def\mfp{\mathfrak p}

\def\pol{\operatorname{pol}}

\def\Ass{\operatorname{Ass}}

\def\MinAss{\operatorname{MinAss}}
\def\height{\operatorname{height}}
\def\depth{\operatorname{depth}}
\def\deg{\operatorname{deg}}

\def\reg{\operatorname{reg}}
\def\gcd{\operatorname{gcd}}
\def\lcm{\operatorname{lcm}}
\def\v{\operatorname{v}}
\def\supp{\operatorname{supp}}
\def\np{\operatorname{NP}}
\def\conv{\operatorname{conv}}
\def\gr{\operatorname{gr}}

\newtheorem{Theorem}{Theorem}[section]
\newtheorem{Lemma}[Theorem]{Lemma}
\newtheorem{Corollary}[Theorem]{Corollary}
\newtheorem{Proposition}[Theorem]{Proposition}
\newtheorem{Remark}[Theorem]{Remark}

\newtheorem{Example}[Theorem]{Example}

\let\epsilon\varepsilon
\let\phi=\varphi
\let\kappa=\varkappa

\def \mod {\operatorname{mod}}

\begin{document}
	\title{$\v$-numbers of integral closure filtrations of monomial ideals}
	
	\author{Vanmathi A}
	\address{Vanmathi A, 
		Department of Mathematics,
		Indian Institute of Technology, Palakkad, India}
	\email{vanmathianandarajan@gmail.com, 212404007@smail.iitpkd.ac.in}
	
	\author{Parangama Sarkar}
	\thanks{The second author was partially supported by SERB POWER Grant with Grant No. SPG/2021/002423.}
	\address{Parangama Sarkar,  Department of Mathematics,
		Indian Institute of Technology, Palakkad, India}
	\email{parangamasarkar@gmail.com, parangama@iitpkd.ac.in}
	\begin{abstract}
		In this article, we investigate the $\v$-numbers of powers of monomial ideals and their integral closures in a polynomial ring $S$. We provide an alternative proof for determining the $\v$-numbers of powers of complete intersection monomial ideals. Furthermore, we analyze the  $\v$-numbers associated to integral closure filtrations of irreducible monomial ideals and explore their relationship with the Castelnuovo–Mumford regularity of these ideals. Consequently, we obtain that for all $n\geq 1$,       $\reg(S/\overline{I^n})=\v(\overline{I^n})=n\alpha(I)-1$ where $I$ is an equigenerated irreducible monomial ideal. We also show that, for any integer $p\geq 3$, there exists a squarefree monomial ideal $I$ (not necessarily equigenerated) in a polynomial ring $R$ such that $\big(\dim R/I-\depth R/I\big)-\v(I)=p-2$ which provides a negative answer to the question \cite[Question 5.5]{SS}. Finally, we give an upper bound for $\v$-numbers associated to the integral closure filtrations of complete intersection monomial ideals and explicitly compute these $\v$-numbers in certain cases. As a consequence, we show that for any integer $a\geq 1$, there exists a height two equigenerated complete intersection monomial ideal $I$ such that $\reg(S/\overline{I^n})-\v(\overline{I^n})=a-1$ for all $n\geq 1$. Moreover, we establish that for complete intersection monomial ideals, the  $\v$-numbers of powers of ideals can be arbitrarily larger than the  $\v$-numbers of integral closures of their powers.
	\end{abstract}	
	\maketitle
	\section{introduction}
	Let $S=K[x_1,\ldots,x_m]$ be a polynomial ring over a field $K$ and $I$ be a homogeneous ideal in $S$. For any $\mfp\in\Ass(I)$, there exists a homogeneous element $f\in S$ such that $\mfp=(I:f)$ \cite[Lemma 1.5.6]{BH}. In \cite{CSTPV}, Cooper {\it et al.}, defined a new invariant associated to $I$, known as $\v$-number of $I$, given by $$\v(I): =\min\{u: \mbox{there exists } f\in S_u \mbox{ and }\mfp\in \Ass(I)\mbox{ such that }\mfp=(I:f)\}.$$ This invariant is named in honor of the mathematician Wolmer Vasconcelos and was first introduced to study the asymptotic behaviour of the minimum distance of projective Reed–Muller-type codes  \cite{CSTPV}. For any $\mfp\in \Ass(I)$, the local $\v$-number of $I$ at $\mfp$ is defined as,
	$\v_\mfp(I):=\min\{u: \mbox{there exists } f\in S_u \mbox{ such that }\mfp=(I:f) \}$ and by definition $\v(I)=\min\{\v_\mfp(I): \mfp\in\Ass(I)\}$ \cite{CSTPV}. The invariant, $\v$-number of an ideal, attracted considerable attention from researchers working at the intersection of Commutative Algebra and Combinatorics. This led to the discovery of numerous connections with other important invariants, especially in graph theory, enriching the understanding of both algebraic and combinatorial structures (see \cite{F23},  \cite{GRV}, \cite{JV}, \cite{JVS},  \cite{KSa}, \cite{SS} and the references therein). The $\v$-numbers of homogeneous ideals are also linked with the Castelnuovo-Mumford regularity of those ideals (\cite{AS},  \cite{JV}, \cite{KSa}, \cite{SS},\cite{Y}).
	\\ Lately, there has been a growing interest in exploring the asymptotic behaviour of the $\v$-numbers associated to homogeneous ideals (see \cite{AS}, \cite{Co}, \cite{FS},  \cite{FSmarch} \cite{MRK}). In \cite{Co}, Conca proved that the invariant $\v(I^n)$ is eventually a linear polynomials in $n$ with integer coefficients. Recently, in \cite{AS}, the authors extended this result for filtrations and proved that  the $\v$-invariant associated to an integral closure filtartion  $\{\overline{I^n}\}$ of homogeneous ideals is also eventually a linear polynomial in $n$. While the study of $\v$-numbers for ordinary powers of ideals has been extensively developed, extending these investigations to integral closure filtrations presents significantly greater difficulties. Unlike the case of ordinary powers, the behaviour of integral closures is more subtle. Despite the combinatorial structure of monomial ideals, their integral closures often exhibit considerably more complex algebraic properties. Even when the generators of a monomial ideal are explicitly known, determining the generators of its integral closure remains a challenging problem, both from computational and theoretical perspectives.
	\\The objective of this paper is to analyze the $\v$-numbers of powers of ideals and the integral closures of powers of ideals. Among the themes presented are: $(1)$ the computation of the  $\v$-numbers of powers of complete intersection monomial ideals and $(2)$ the study of the  $\v$-numbers  associated to integral closure filtrations of complete intersection monomial ideals with a focus on elucidating their relationship with the Castelnuovo–Mumford regularity. In Section \ref{section2},  we recall the required definitions and review several well-known results on monomial ideals which we use throughout the paper. For a squarefree complete intersection monomial ideal $I$ with the minimal monomial generating set $\mathcal G(I)=\{u_1,\ldots,u_r\}$, Jaramillo and Villarreal  \cite{JV} showed that $\v(I)=\deg u_1+\cdots+\deg u_r-r$. Extending this result via polarization techniques, Saha and Sengupta proved that the formula holds for arbitrary complete intersection monomial ideals \cite{SS}. In \cite{BM}, $\v$-numbers of powers of irreducible monomial ideals were computed. More recently, the $\v$-numbers of powers of complete intersection monomial ideals were investigated in \cite{FM}. In Section \ref{section3}, we provide an alternative proof of the above result. As part of our approach, we show that if $I$ is a complete intersection monomial ideal and  $m\in S$ is any monomial satisfying $(I:m)=P$ for some  $P\in\Ass(I)$, then there exists a fixed principal  monomial ideal $\langle g_P\rangle$ in $S$ (depending on $P$) such that $m\in \langle g_P\rangle$ and $\deg g_P=\v(I)$. 
	\begin{Theorem}{\em(Theorem \ref{ci})} 
		Let $I=\langle u_1,\ldots u_r\rangle$ be a complete intersection monomial ideal in $S$ and $P=\langle y_1,\dots,y_r\rangle$ be any associated prime of  $I$ where $y_1,\ldots,y_r\in\{x_1,\ldots,x_m\}$.  Let $\displaystyle g={\prod\limits_{i=1}^r {u_i}}/{\prod\limits_{i=1}^r{y_i}}.$
		\begin{enumerate}
			\item[$(1)$] If $m$ is a monomial such that $(I^n:m)=P$ then $\displaystyle g\mid m$ and $\displaystyle\deg{\frac{m}{g}}\geq (n-1)\alpha(I).$
			\item[$(2)$] For all $n\geq 1$, $\v(I^n)=\v_P(I^n)=n\alpha(I)+\v(I)-\alpha(I)$ where $\v(I)=\sum\limits_{i=1}^r \deg u_i-r$.  
		\end{enumerate}		
	\end{Theorem}	
	We start Section \ref{section4}, by showing that for any homogeneous ideal $I$,   $\displaystyle\lim\limits_{n\to\infty}\frac{\alpha(\overline{I^n})}{n}=\alpha(I)$ (see Lemma \ref{G}) extending the result in \cite[Corollary 3.8]{GSV}. 
	Next, we focus on investigating the $\v$-numbers associated to integral closure filtrations of complete intersection ideals and explore their relationship with the Castelnuovo–Mumford regularity of these ideals.  Let $I$ be an unmixed monomial ideal in $S$ with a unique irredundant decomposition $I=\bigcap_{i=1}^sQ_i$ where $Q_i$'s are irreducible monomial ideal. Suppose $I$ is a Simis ideal (i.e. $I^n=I^{(n)}$ for all $n\geq 1$). Then  by \cite[Theorem 4.1]{GMBV}, ${\overline{I^n}}=\bigcap_{i=1}^h{\overline{Q_i^n}}$ for all $n\geq 1$. Let $P\in \Ass(\overline{I^n})$. Then for any monomial $f\in S$ with $\bigcap_{i=1}^s(\overline{Q_i^n}:f)=(\overline{I^n}:f)=P$, we have $(\overline{Q_i^n}:f)=P$ for some $1\leq i\leq s$ (as $I$ is unmixed) and thus $\v(\overline{Q_i^n})\leq \v_P(\overline{I^n})$. This motivates us to study the $\v$-numbers associated to  integral closure filtrations of irreducible monomial ideals and we prove the following.
	\begin{Theorem}{\em(Theorem \ref{gen})} 
		Let $I=\langle x_{i_1}^{a_1},\ldots,x_{i_k}^{a_k}\rangle$ be a monomial ideal in $S$ such that $1\leq a_1\leq \cdots\leq a_k$ and $k\geq 2$. Then for all $n\geq 1$,
		\begin{enumerate}
			\item[$(1)$]   $\displaystyle n\alpha(I)
			+\lceil \frac{a_2}{\alpha(I)}-\frac{a_2}{\delta(I)}\rceil-1\leq \v(\overline {I^n})\leq n\alpha(I)+\lceil \delta(I)/\alpha(I)\rceil-2$ for all $n\geq 1$.
			\item[$(2)$] Suppose there exists $j\geq2$ such that $a_1=\cdots=a_{j-1}\leq a_j=\cdots=a_k$, then $\displaystyle\v(\overline{I^n})=n\alpha(I)+\lceil\frac{\delta(I)}{\alpha(I)}\rceil-2$.
			\item[$(3)$]  $\v(\overline {I^n})= n\alpha(I)+\lceil \delta(I)/\alpha(I)\rceil-2$ when $k=2$.
			\item[$(4)$]  
			\[\displaystyle0\leq \reg(S/\overline {I^n})-\v(\overline {I^n})\leq \left\{
			\begin{array}{l l}
				~~~~~\displaystyle (\delta(I)-\alpha(I))n+\dim S/I& \quad \text{if $a_2=\delta(I)$ }\\ \vspace{0.3mm}\\
				\displaystyle (\delta(I)-\alpha(I))n+\dim S/I-1& \quad \text{if $a_2<\delta(I)$ . }\\ 
			\end{array} \right.\] 
			In particular, if $I$ is an equigenerated irreducible monomial ideal then for all $n\geq 1$, $\reg(S/\overline {I^n})=\v(\overline {I^n})=n\alpha(I)-1$. 
		\end{enumerate}
	\end{Theorem}
	As a consequence, we compute $\v$-numbers of integral closures of powers of an edge ideal $I(D)$ of a weighted oriented graph $D$ whose underlying graph is a complete bipartite graph and  show that $\v(\overline{I(D)^n})=n\alpha(I(D))-1$ for all $n\geq 1$ (Corollary \ref{WOG}). This leads us to prove Corollary \ref{negativeanser}, which asserts that for any integer $p\geq 3$, there exists a squarefree monomial ideal $I$ (not necessarily equigenerated) in a polynomial ring $R$ such that $\big(\dim R/I-\depth R/I\big)-\v(I)=p-2$ and it provides a negative answer to the question posed in \cite[Question 5.5]{SS}. Our next aim is to analyze the integral closure filtrations of height three irreducible monomial ideals, for which we are able to explicitly compute the corresponding $\v$-numbers in most instances.
		\begin{Theorem}{\em(Theorem \ref{threegen})} 
		Let $I=\langle x_{i_1}^{a_1},x_{i_2}^{a_2},x_{i_3}^{a_3}\rangle$ be a monomial ideal in $S$ with $1\leq a_1\leq a_2\leq a_3$. 
		\begin{enumerate}
			\item[$(1)$] $\v(\overline{I})=\min\limits_{1\leq m\leq a_1} \deg f_m$.
			\item[$(2)$] If $a_1=1$ then $\displaystyle\v(\overline{I^n})=n+a_2+\lceil\frac{a_3}{a_2}\rceil-3$ for all $n\geq 1$.
					\item[$(3)$] Let $a_1\geq 2$.
			\begin{enumerate}
				\item[$(i)$] $\displaystyle\v(\overline{I})=\min\limits_{1\leq m\leq a_1-1} \deg f_m$.
				\item[$(ii)$] Let $\v(\overline{I})=\deg f_l$ for some $1\leq l\leq a_1-1$. Then for all $n\geq 1$, $$\displaystyle\v(\overline{I^n})=(n-1)a_1+\deg f_l=na_1-l+\lceil\frac{la_2}{a_1}\rceil+\lceil\frac{a_3}{a_2}(\frac{l a_2}{a_1}-\lceil\frac{la_2}{a_1}\rceil+1)\rceil-2.$$
			\item[$(iii)$] If $a_2\equiv0~(\mod a_1)$ or $a_2\equiv1~(\mod a_1)$, then $\v(\overline{I})=\deg f_1$ and for all $n\geq 1$,
			$$\displaystyle\v(\overline{I^n})=(n-1)a_1+\deg f_1=na_1+\lceil\frac{a_2}{a_1}\rceil+\lceil\frac{a_3}{a_2}(\frac{ a_2}{a_1}-\lceil\frac{a_2}{a_1}\rceil+1)\rceil-3.$$
			\item[$(iv)$] $\displaystyle\lceil a_2/a_1\rceil-1\leq \lceil a_2/a_1-a_2/a_3\rceil\leq \lceil a_2/a_1\rceil$ and if $\displaystyle\lceil a_2/a_1\rceil-1=\lceil a_2/a_1-a_2/a_3\rceil$ then for all $n\geq 1$, $\displaystyle\v(\overline{I^n})=na_1+\lceil a_2/a_1\rceil-2.$  
			\item[$(v)$] If $\displaystyle\lceil\frac{(a_1-1)a_3}{a_1a_2}\rceil=1$ then $\displaystyle\v(\overline{I^n})=(n-1)a_1+\deg f_1$ for all $n\geq 1$.
			\end{enumerate}	
		\end{enumerate}
		Here $f_m$ are the monomials in Proposition \ref{monomialofvint}.
	\end{Theorem}
	Finally, we show that $\v$-numbers of integral closure filtrations of complete intersection ideals are bounded above by the $\v$-numbers of powers of those ideals. In particular, computations of $\v$-numbers of integral closure filtrations of $\height 2$ complete intersection ideals for certain cases allow us to show that for any integer $a\geq 1$, there exists a height two equigenerated complete intersection monomial ideal $I$ such that $\reg(S/\overline{I^n})-\v(\overline{I^n})=a-1$ for all $n\geq 1$.
		\begin{Theorem}{\em(Theorem \ref{uppbndcompleteint})}
		Let $I=\langle u_1,\dots,u_r\rangle$ be a complete intersection monomial ideal in $S$ with $\deg u_1=\alpha(I)$. Then the following hold.
		\begin{itemize} 
			\item[$(1)$] 
			\begin{enumerate}
				\item[$(i)$] $\v(\overline{I^n})\leq\v(I^n)$ for all $n\geq 1$.
				\item[$(ii)$] If $u_1$ and $u_t$ are not squarefree for some $2\le t\le r$, then $\v(\overline{I^n})\le\v(I^n)-1$ for all $n\geq 1$.
				\item[$(iii)$] Let $I$ be an equigenerated complete intersection ideal with $u_i=m_i^\alpha$  where $m_i$ are squarefree monomials for all $1\leq i \leq r$. Then for all $n\geq 1$, $$\reg(S/\overline{I^n})=\v(\overline{I^n})=n\alpha(I)+(|\supp(u_1)|-1)(r-1)-1.$$
			\end{enumerate}	
			\item[$(2)$] Suppose $r=2$, $u_1=(x_{i_1}\cdots x_{i_q})^\alpha$ and $u_2=x_{j_1}^{\beta_1}\cdots x_{j_l}^{\beta_l}$ with $1\leq \alpha\leq\beta_1\leq\cdots\leq\beta_l$ and $q\leq l$. Then  for all $n\geq 1$ and $P\in \Ass(\overline{I^n})$, $$\displaystyle\v(\overline{I^n})=\v_P(\overline{I^n})=n\alpha(I)+\sum\limits_{j=1}^l\lceil\frac{\beta_j}{\alpha}\rceil-2.$$ In particular, if $\alpha=1$ then  $\v(I^n)=\v(\overline{I^n})$ for all $n\geq 1$.
			\item[$(3)$] Suppose $r=2$, $u_1=x_{i_1}^\alpha$, $u_2=x_{j_1}^{\beta_1}\cdots x_{j_l}^{\beta_l}$ with $l\geq 2$ and $\deg u_2=\alpha$. Then $\v(\overline{I^n})\geq n\alpha(I)$ for all $n\geq 1$. 
			\\Moreover, if there exists $j$ with $2\beta_j\geq \alpha$ then $\v(\overline{I^n})=n\alpha(I)$ for all $n\geq 1$.\vspace{2mm}
			\item[$(4)$] For any integer $a\geq 1$, there exists a height two equigenerated complete intersection monomial ideal $I$  such that $\reg(S/\overline{I^n})-\v(\overline{I^n})=a-1$ for all $n\geq 1$.
		\end{itemize}
	\end{Theorem}
	We conclude Section 4 by showing that, for complete intersection monomial ideals, the  $\v$-numbers of powers of ideals can be arbitrarily larger than the  $\v$-numbers of integral closures of their powers.
	\begin{Corollary}{\em{(Corollary \ref{diff})}} For any integer $q\geq 0$, there exist  an equigenerated complete inetersection monomial ideal $I$ and a non-equigenerated complete inetersection monomial ideal $J$ such that for all $n\geq 1$, $$\v(I^n)-\v(\overline{I^n})=q=\v(J^n)-\v(\overline{J^n}).$$ 
	\end{Corollary}	
	\section{preliminaries}\label{section2}
		Let $I$ be an ideal in a ring $R$. An element $x\in R$ is called integral over $I$ if there exists an integer $m\geq 1$ and elements $a_i\in I^i$ for $1\leq i\leq m$ such that $$x^m+a_1x^{m-1}+\cdots+a_{m-1}x+a_m=0.$$ The integral closure of $I$, denoted by $\overline{I}$, is an ideal in $R$, which is the set of all elements in $R$ that are integral over $I$ \cite[Corollary 1.3.1]{SH}. If $I$ is a monomial ideal in a polynomial ring over a field then $\overline{I}$ is also a monomial ideal \cite[Proposition 1.4.2]{SH}. 
	\\Throughout this article, $S=K[x_1,\ldots,x_m]$ denote a polynomial ring over a field $K$ and $\mathbb N$ denote the set of all nonnegative integers. 
	For any ${\bf {a}}=(a_1,\ldots,a_m)\in\mathbb N^m$, $x^{\bf {a}}$ represents the monomial  $x_1^{a_1}\cdots x_m^{a_m}$. We denote the degree of a monomial $x^{\bf {a}}$ by $\deg x^{\bf {a}}$. Let $f$ be a monomial in $S$. Then $\supp(f):=\{x_i: 1\leq i\leq m\mbox{ and }x_i\mbox{ divides } f\}$. An ideal in $S$ is called a monomial ideal if it is generated by monomials. 
		\begin{Remark}\label{powers}{\em
			Let $I=\langle u_1,\ldots,u_r\rangle$ be a monomial ideal in $S$. For any $m\geq 1$, consider the monomial ideals $J_m=\langle u_1^m,\ldots,u_r^m\rangle$ in $S$. Then by \cite[Lemma 2.5]{RRV}, for any $m\geq 1$, we have $\np(I^m)=\np(J_m)$ and thus $\overline{I^m}=\overline{J_m}$.}
	\end{Remark}	
	We recall the following properties of monomial ideals that we use in our results.
	\begin{enumerate}
		\item[$(P1)$] A monomial ideal $I$ has a unique minimal monomial set of generators \cite[Proposition 1.1.6]{HH} and we denote it by $\mathcal G(I)$. For any monomial ideal $I$, we define $\alpha(I):=\min\{d: I_d\neq 0\}$ and $\delta(I):=\max\{\deg u: u\in \mathcal G(I)\}$. 
		\\A monomial ideal $I$ is said to be generated by pure powers of variables if $\mathcal G(I)=\{x_{i_1}^{a_{i_1}},\ldots x_{i_r}^{a_{i_r}}\}$ for some ${i_1},\ldots,{i_r}\in\{1,\ldots,n\}$.
		\\A monomial ideal $I$ is called a complete intersection monomial  ideal if $\height(I)=\mu(I)$ where $\mu(I)$ is the cardinality of $\mathcal G(I)$.
		\item[$(P2)$] If $I$ and $J$ are monomial ideals then the monomial ideal $I\cap J$ is generated by the set $\{\lcm(u,v):u\in\mathcal G(I), v\in \mathcal G(J)\}$ \cite[Proposition 1.2.1]{HH}.
		\item[$(P3)$] If $I$ is a monomial ideal and $v$ is a monomial then the ideal $(I:v)$ is generated by the set $\{u/\gcd(u,v):u\in\mathcal G(I)\}$ \cite[Proposition 1.2.2]{HH}. 
		\item[$(P4)$] Let $I=\langle
		x_{i_1}^{a_{i_1}},\ldots x_{i_r}^{a_{i_r}} \rangle$ be a monomial ideal generated by pure powers of variables. Then $I^n$, $\overline{I^n}$ are  $\langle{x_{i_1},\ldots x_{i_r}} \rangle$-primary ideals for all $n\geq 1$  \cite[Proposition 6.1.7 and Proposition 12.1.4]{MAVill}. By Remark \ref{powers}, In particular, for any irreducible monomial ideal $I=\langle x_{i_1}^{a_{i_1}},\ldots x_{i_r}^{a_{i_r}}\rangle$, we have $\overline{I^n}=\overline{\langle x_{i_1}^{na_{i_1}},\ldots x_{i_r}^{na_{i_r}}\rangle}$ for all $n\geq 1$.
		\item[$(P5)$]  Let $I$ be a monomial ideal in $S$. Then $I=\bigcap\limits_{i=1}^h Q_i$ for some monomial ideals $Q_i$ which are generated by pure powers of variables \cite[Theorem 1.3.1]{HH}. If none of the ideals $Q_i$ can be omitted from the above representation, i.e., the decomposition is irredundant, then the representation is unique. 
		\end{enumerate}
		The next result follows from \cite[Proposition 3.3]{DMGV}. We provide a simpler proof that aligns more naturally with our framework.
	\begin{Lemma}{\label{membership}}
		Let $I=\langle x_{i_1}^{b_{i_1}},\ldots x_{i_r}^{b_{i_r}}\rangle$ be a monomial ideal in $S$ generated by pure powers of variables. Then a monomial $x_1^{a_1}\cdots x_n^{a_n}\in \overline I$ if and only if  $\displaystyle \sum\limits_{l=1}^r \frac{a_{i_l}}{b_{i_l}}\geq 1$.
		\end{Lemma}	
		\begin{proof}
			Let ${\bf{a}}=(a_1,\ldots,a_n)$ and for all $1\leq i\leq n$, $e_i=(0,\ldots,0,1,0,\ldots,0)$ where $1$ is at $i$-th position. Now $x^{\bf{a}}\in \overline I$ implies ${\bf{a}}\in \np(I)=\conv(b_{i_1}e_{i_1},\ldots b_{i_r}e_{i_r})+\mathbb R^n_{\geq 0}$. Therefore there exist $\lambda_{l}\in \mathbb R$ such that $0\leq \lambda_l\leq 1$  for $1\leq l\leq r$ with $\sum\limits_{l=1}^r \lambda_l=1$ such that $(a_{i_1},\ldots,a_{i_r})\geq (\lambda_1{b_{i_1}},\ldots, \lambda_r{b_{i_r}})$. Hence $\displaystyle \sum\limits_{l=1}^r \frac{a_{i_l}}{b_{i_l}}\geq \sum\limits_{l=1}^r \lambda_l=1.$
				Suppose $\displaystyle \sum\limits_{l=1}^r \frac{a_{i_l}}{b_{i_l}}\geq 1.$ Let $\displaystyle \sum\limits_{l=1}^r \frac{a_{i_l}}{b_{i_l}}=1+L$ for some $L\in\mathbb R_{\geq 0}$. If $r=1$ then $a_{i_1}\geq b_{i_1}$ and hence $x_{i_1}^{a_{i_1}}\in I\subset \overline I$. Therefore $x^{\bf{a}}\in \overline I$. Suppose $r\geq 2$. If ${a_{i_l}}\geq b_{i_l}$ for some $1\leq l\leq r$ then $x_{i_1}^{a_{i_1}}\cdots x_{i_r}^{a_{i_r}}\in \overline I$ and hence $x^{\bf{a}}\in \overline I$. Therefore we assume that $\displaystyle 0\leq  \frac{a_{i_l}}{b_{i_l}}< 1$ for all $1\leq l\leq r$.  Consider the continuous function $\displaystyle\phi: \prod\limits_{l=1}^r[0,\frac{a_{i_l}}{b_{i_l}}]\rightarrow [0,1+L]$ defined by $\phi(x_1,\ldots,x_r)=\sum\limits_{l=1}^rx_l$. Since $\phi(0,\ldots,0)=0$ and $\phi(\frac{a_{i_1}}{b_{i_1}},\ldots,\frac{a_{i_r}}{b_{i_r}})=1+L$, we have $\phi^{-1}(1)\neq\emptyset$.  Let $\mu_{l}\in \mathbb R$ such that $\displaystyle 0\leq \mu_{l}\leq \frac{a_{i_l}}{b_{i_l}}$ for all $1\leq l\leq r$ and $\displaystyle \sum\limits_{l=1}^r (\frac{a_{i_l}}{b_{i_l}}-\mu_l)=1.$ Therefore
			\begin{eqnarray*}\sum\limits_{p=1}^ra_{i_p}e_{i_p}&=&\sum\limits_{p=1}^r(\frac{a_{i_p}}{b_{i_p}}-\mu_p)b_{i_p}e_{i_p}+\sum\limits_{p=1}^r\mu_pb_{i_p}e_{i_p}\in \np(I).\end{eqnarray*}
			Hence $x_{i_1}^{a_{i_1}}\cdots x_{i_r}^{a_{i_r}}\in \overline I$ and thus $x^{\bf{a}}\in \overline I$.
		\end{proof}	
		
			\begin{Lemma}\label{least}
			Let $I$ be a monomial ideal in $S$ and $P\in \Ass(I)$. Suppose $f\in S$ is a monomial  such that $(I:f)=P$ and $\v_P(I)=\deg f$. Then $\supp(f)\subseteq \bigcup\limits_{u\in \mathcal{G}({I})}\supp(u)$.
		\end{Lemma}
		\begin{proof}
			Suppose there exists $x\in \supp(f)\setminus\bigcup\limits_{u\in \mathcal{G}({I})}\supp(u)$. Let $m=f/x$. Then $P=(I:f)=((I:x):m)=(I:m)$ which contradicts that $\v_P(I)=\deg f$.
		\end{proof}
		For an ideal $I$ in $S$, the $n$-th symbolic power of $I$ is defined in two ways in literature:
		\begin{enumerate}
			\item $I^{(n)}=\bigcap\limits_{Q\in\MinAss(I)}I^nS_Q\cap S$,
			\item 	$I^{\langle n\rangle}=\bigcap\limits_{Q\in\Ass(I)}I^nS_Q\cap S$.	
		\end{enumerate}	
		Note that $I^n\subseteq I^{\langle n\rangle}\subseteq I^{(n)}$ for all $n\geq 1$. If $I$ has no embedded primes then $I^{\langle n\rangle}= I^{(n)}$ for all $n\geq 1$.
		\\Let $M$ be a finitely generated graded module over a polynomial ring $S$. Consider a  minimal graded free resolution of $M$ as $S$-module
		$$0\rightarrow \bigoplus\limits_{j}S(-j)^{\beta_{pj}}\rightarrow \bigoplus\limits_{j}S(-j)^{\beta_{(p-1)j}}\rightarrow\cdots\rightarrow\bigoplus\limits_{j}S(-j)^{\beta_{1j}}\rightarrow \bigoplus\limits_{j}S(-j)^{\beta_{0j}}\rightarrow 0.$$ Then the Castelnuovo-Mumford regularity of $M$, denote by $\reg(M)$, is defined as $$\reg(M):=\max\{j-i: \beta_{ij}\neq 0, 0\leq i\leq p\}.$$  For any homogeneous ideal $I$ in $S$, we have $\reg(S/I)=\reg(I)-1$.
		\subsection*{Graph} 
		A finite simple graph $G$ with the vertex set $V(G)$ and the edge set $E(G)$ is called a bipartite graph if the vertex set $V(G)$ can be partitioned into two  sets $A$, $B$ and each edge is of the form $(i,j)\in E(G)$ such that $i\in A$ and $j\in B$. A bipartite graph $G$ with a vertex set $V(G)=A\sqcup B$ is called a complete bipartite graph if every vertex of $A$ is adjacent to every vertex of $B$. We denote a complete bipartite graph by $K_{m,n} $ where the vertex set $V(G)$ is partitioned into two sets $A$ and $B$ with $|A|=m\leq n=|B|$. A bipartite graph $G$ is called a complete bipartite graph if $E(G)=\{(a,b): a\in A, b\in B \}$.
	\\Let $G$ be a finite simple graph with the vertex set $V(G)$ and the edge set $E(G)$. A weighted oriented graph $D$ with underlying graph $G$, is a graph $D$ with a triplet $(V(D), E(D),w)$ where $V(D)=V(G)$, $E(D)\subset E(G)\times E(G)$ with $|E(G)|=|E(D)|$ and $E(G)$ is a set of pairs $\{x_i,x_j\}$ such that $(x_i,x_j)\in E(D)$ and $w:V(D)\rightarrow \ZZ_{>0}$ is a weight function. The sets $V(D)=\{x_1,\ldots,x_n\}$ and $E(D)$ are called the vertex set and the edge set of $D$ repectively. Consider a polynomial ring $K[x_1,\ldots,x_n]$ over a field $K$. The edge ideal of $D$ is defined as $I(D):=\{x_ix_j^{w(x_j)}: (x_i,x_j)\in E(D)\}$.
\section{$\v$-numbers of powers of complete intersection ideals}{\label{section3}}
In this section, we show that for a complete intersection monomial ideal $I=\langle u_1,\ldots,u_r\rangle$, $\v(I^n)=\v_P(I^n)=n\alpha(I)+\v(I)-\alpha(I)$ for all $P\in\Ass(I)$ and  $n\geq 1$ where $\v(I)=\sum\limits_{i=1}^r \deg u_i-r$. As mentioned in the introduction, this result is recently proved in \cite{FM}. Here, we provide an alternative proof and  in the process, we show that every monomial $m\in S$ satisfying $(I:m)=P$ for some  $P\in\Ass(I)$, belongs to a fixed principal  monomial ideal $\langle g_P\rangle$ in $S$ depending on $P$ and $\deg g_P=\deg g_Q$ for all $P,Q\in \Ass(I)$.
\\Let $I$ be a monomial complete intersection ideal of $\height(I)=r$ and $\mathcal G(I)=\{u_1,\ldots,u_r\}$. We recall some well-known results on complete intersection monomial ideals that we use in our proofs. 
	\begin{enumerate}
\item[$(i)$] $\{u_1,\ldots,u_r\}$ is a monomial regular sequence and $\gcd(u_i,u_j)=1$ for all $i\neq j$.
\item[$(ii)$] Let $I=Q_1\cap\cdots\cap Q_s$ be the unique irredundant decomposition of $I$ where $Q_i$ are monomial ideals generated by pure powers of variables for all $1\leq i\leq s$. Then $\sqrt{Q_i}\neq \sqrt{Q_j}$ for all $i\neq j$ and $\height(\sqrt{Q_i})=r$ for all $1\leq i\leq s$.
\item[$(iii)$] For all $n\geq 1$, $I^{(n)}=I^n$. Since $I$ is generated by a regular sequence, the graded associated ring of $I$, $\gr(I)=\bigoplus\limits_{j\geq 0}I^j/I^{j+1}$, is isomorphic to a polynomial ring over $S/I$. Hence $I^j/I^{j+1}$ is a free $S/I$-module for all $j\geq 0$. Therefore using the short exact sequence of $S$-modules
$$0\longrightarrow I^j/I^{j+1}\longrightarrow S/I^{j+1}\longrightarrow S/I^{j}\longrightarrow 0,$$  for all $j\geq 1$, we get
	\begin{enumerate}
\item[$(a)$] $\Ass(S/I^2)=\Ass(S/I)$, 
\item[$(b)$] $\Ass(S/I)\subset\Ass(S/I^3)\subset\Ass(S/I^2)\cup \Ass(S/I)=\Ass(S/I)$ and continuing this process, we have $\Ass(S/I^n)=\Ass(S/I)$ and hence $I^{(n)}=I^n$ for all $n\geq 1$.
\end{enumerate}	
\end{enumerate}	
\begin{Lemma}\label{In:u=I}
	Let $I=\langle u_1,\ldots u_r\rangle$ be a complete intersection monomial ideal in $S$. Then for any integers $b_i\geq 0$ for $1\leq i\leq r$ with $\sum_{i=1}^r b_i\leq n$,
	 $\displaystyle(I^n: u_1^{b_1}\cdots u_r^{b_r})=I^{n-\sum\limits_{i=1}^r b_i}$. 
\end{Lemma}
\begin{proof}
	We first show that $(I^n:u_i)=I^{n-1}$ for all $n\geq 1$ and $1\leq i\leq r$. Since $u_i\in I$, we have $I^{n-1}\subseteq (I^n:u_i)$ for all $1\leq i\leq r$. We show that $(I^n:u_i)\subseteq I^{n-1}$. Suppose $f\in (I^n:u_i)$, then $fu_i\in I^n$. Therefore $fu_i= h u_1^{c_1}\cdots u_r^{c_r}$ for some monomial $h\in S$ and $\sum\limits_{j=1}^r c_j=n$ with $c_j\geq 0$ for $1\leq j\leq r$.  Since $\gcd(u_i,u_j)=1$ for $i\neq j$, we have $u_i\mid h u_i^{c_i}$. If $c_i=0$, then $f=\displaystyle \frac{h}{u_i} u_1^{c_1}\cdots u_r^{c_r}\in I^n\subset I^{n-1}$.  If $c_i\geq 1$, then $f= h u_1^{c_1}\cdots u_i^{c_i-1}\cdots u_r^{c_r}\in I^{n-1}$.  Therefore $(I^n:u_i)=I^{n-1}$ for all $n\geq 1$ and $1\leq i\leq r$.  
	 Repeating this process $b_i$ times we get $(I^n: u_i^{b_i})=I^{n-b_i}$.
	Using this repeatedly, we get 
	$(I^n: u_1^{b_1}\cdots u_r^{b_r})=(I^{n-b_1}:u_2^{b_2}\cdots u_r^{b_r})=\cdots=(I^{(n-\sum\limits_{i=1}^{r-1} b_i)}:u_r^{b_r})=I^{n-\sum\limits_{i=1}^r b_i}.$
\end{proof}
In \cite[Proposition 3.10]{SS}, using polarization and \cite[Proposition 3.9]{JV}, Saha and Sengupta showed that if $I=\langle u_1,\ldots u_r\rangle$ is a complete intersection monomial ideal then $\v(I)=\sum\limits_{i=1}^r \deg u_i-r$.  In the following result, we provide an alternative proof of this by explicitly constructing a monomial $f$ satisfying $(I:f)=P$ and $\deg f=\v(I)$ that also shows that $\v(I)=\v_P(I)$ for all $P\in \Ass(I)$. The proof also provides a method to compute $\v$-numbers of primary ideals.
\begin{Proposition}{\label{idealcase}}
Let $I=\langle u_1,\ldots u_r\rangle$ be a  monomial ideal and $P=\langle y_1,\ldots,y_r\rangle$ be any associated prime of $I$ where $y_1,\ldots,y_r\in\{x_1,\ldots,x_m\}$. 
	\begin{enumerate}
\item[$(1)$] Suppose $I$ is a complete intersection ideal and $\displaystyle f={\prod\limits_{i=1}^r {u_i}}/{\prod\limits_{i=1}^r{y_i}}.$ Then $(I:f)=P$ and
$\v(I)=\v_P(I)=\sum\limits_{i=1}^r \deg u_i-r$.  
\item[$(2)$] Suppose $I$ is a $P$-primary monomial ideal and $I=\bigcap\limits_{i=1}^s Q_i$ is the unique irredundant decomposition of $I$ where $Q_i$ is a monomial ideal generated by pure powers of variables for all $1\leq i\leq s$. Then $\v(I)=\min\{\v(Q_i): 1\leq i\leq s\}$.
		\end{enumerate}
\end{Proposition}	
\begin{proof}
	$(1)$ Suppose $I$ is a complete intersection ideal.  Without loss of generality, we assume $y_i\in\supp(u_i)$ for all $1\leq i\leq r$. Since $\gcd(u_i,u_j)=1$ for all $i\neq j$, we have $\displaystyle\gcd(u_j,\prod\limits_{i=1}^r\frac{u_i}{y_i})=\frac{u_j}{y_j}$ and hence $\displaystyle \frac{u_j}{\gcd(u_j,f)}={y_j}$. Therefore by the property $(P3)$, we have $(I:f)=P.$ Therefore $\v_P(I)\leq \deg f=\sum\limits_{i=1}^r \deg u_i-r$.
	
	Let $g$ be a monomial such that $(I:g)=P$ and $\deg g=\v_P(I)$. Then for all $1\leq i\leq r$, 
	\begin{eqnarray*}
		g\in (I:P)=\bigcap\limits_{i=1}^r (I:y_i)&=&\bigcap\limits_{i=1}^r\langle u_1,\dots,u_{i-1},\frac{u_i}{y_i},{u_{i+1}}\dots,u_r\rangle\\&\subset& \langle u_1,\dots,u_{i-1},\frac{u_i}{y_i},{u_{i+1}}\dots,u_r\rangle.	\end{eqnarray*}
	Since $g\notin I$, we have $u_i\nmid g$ for all $1\leq i\leq r$. Therefore $\displaystyle \frac{u_i}{y_i}\mid g$ for all $1\leq i\leq r$. Since $\displaystyle \gcd(\frac{u_i}{y_i},\frac{u_j}{y_j})=1$ for all $i\neq j$, we have $f\mid g$. Therefore $\v_P(I)=\deg g\geq \deg f$. Hence for all $P\in\Ass(R/I)$, $\displaystyle\v(I)=\v_P(I)=\deg f=\sum\limits_{i=1}^r \deg u_i-r$. 
	
	$(2)$ By \cite[Proposition 6.1.7]{MAVill}, $Q_i=\langle y_1^{\alpha_{i_1}},\ldots,y_r^{\alpha_{i_r}}\rangle$ with $\alpha_{i_l}\geq 1$ for all $1\leq l\leq r$ and $1\leq i\leq s$ . Let $\v(Q_t)=\min\{\v(Q_i): 1\leq i\leq s\}$ and $g=\prod\limits_{j=1}^ry_j^{\alpha_{t_j}-1}$. Then by part $(1)$, we have $(Q_t:g)=P$ and $\v(Q_t)=\deg g$. Since $Q_t\nsubseteq Q_i$ for any $i\neq t$, there exists $j\in\{1,\ldots r\}$ such that $\alpha_{i_j}<{\alpha_{t_j}}$. Thus  $g\in Q_i$ and $(Q_i:g)=R$ for $i\neq t$. Therefore $(I:g)=P$ and $\v(I)\leq \deg g=\v(Q_t)$.
	
	Let $f\in S$ be a monomial such that $(I:f)=P$ and $\deg f=\v(I)$. Now $P=(I:f)=\bigcap\limits_{i=1}^s (Q_i:f)=(Q_j:f)$ for some $j\in\{1,\ldots,s\}$. Thus $\v(Q_t)\geq \v(I)=\deg f\geq \v(Q_j)\geq \v(Q_t)$. Therefore $\v(I)=\min\{\v(Q_i): 1\leq i\leq s\}$.
\end{proof}	
Now we prove the main result of this section.
\begin{Theorem}\label{ci}
	Let $I=\langle u_1,\ldots u_r\rangle$ be a complete intersection monomial ideal in $S$ and $P=\langle y_1,\dots,y_r\rangle$ be any associated prime of  $I$ where $y_1,\ldots,y_r\in\{x_1,\ldots,x_m\}$. Let $\displaystyle g={\prod\limits_{i=1}^r {u_i}}/{\prod\limits_{i=1}^r{y_i}}.$
	\begin{enumerate}
\item[$(1)$] If $m$ is a monomial such that $(I^n:m)=P$ then $\displaystyle g\mid m$ and $\displaystyle\deg{\frac{m}{g}}\geq (n-1)\alpha(I).$
\item[$(2)$] For all $n\geq 1$, $\v(I^n)=\v_P(I^n)=n\alpha(I)+\v(I)-\alpha(I)$ where $\v(I)=\sum\limits_{i=1}^r \deg u_i-r$.  
\end{enumerate}		
	\end{Theorem}	
\begin{proof}
	Without loss of generality, we assume that $\deg u_1=\alpha(I)$ and $y_i\in\supp(u_i)$ for all $1\leq i\leq r$. 
	
	$(1)$ Let $m$ be a monomial such that $(I^n:m)=P$. Fix any $i\in\{1,\ldots,r\}$. Then $my_i=h u_1^{b_1}\cdots u_r^{b_r}\in I^n$ for some monomial $h\in S$ with $\sum\limits_{j=1}^r b_j=n$. Since $(y_i,u_j)=1$ for $i\neq j$, $b_i=0$ implies  $\prod\limits_{j=1}^r u_j^{b_j}\mid m$ and hence $m\in I^n$ which contradicts that $(I^n:m)=P$. Therefore $b_i\geq 1$ and $\displaystyle  m=h \frac{u_i}{y_i} u_1^{b_1}\cdots u_i^{b_i-1}\cdots u_r^{b_r}$. Hence $\displaystyle \frac{u_i}{y_i}\mid m$. Thus $\displaystyle \frac{u_j}{y_j}\mid m$ for all $1\leq j\leq r$. Since $\displaystyle\gcd(\frac{u_i}{y_i},\frac{u_j}{y_j})=1$ for $i\neq j$, we have $g\mid m$. 
	
	Let $m=ug$ for some monomial $u\in S$. We show that $\deg u\geq (n-1)\alpha(I)$. Since $ug=m\in(I^n:y_1)$, by Lemma \ref{In:u=I}, $\displaystyle  u\prod\limits_{i=2}^r\frac{u_i}{y_i}\in((I^n:y_1):\frac{u_1}{y_1})=(I^n:u_1)=I^{n-1}.$ Hence $\displaystyle  u\prod\limits_{i=2}^r\frac{u_i}{y_i}= h' u_1^{c_1}\cdots u_r^{c_r}$ for some monomial $h'\in S$ with $\sum\limits_{i=1}^r c_i=n-1$. Again using Lemma \ref{In:u=I}, we get, 
\begin{eqnarray*}\displaystyle  P=(I^n:m)=(I^n: u\prod\limits_{i=1}^r\frac{u_i}{y_i})&=&(I^n:h' u_1^{c_1}\cdots u_r^{c_r}\frac{u_1}{y_1})\\&=&((I^n:u_1^{c_1}\cdots u_r^{c_r}):h'\frac{u_1}{y_1})=(I:h'\frac{u_1}{y_1}). \end{eqnarray*} Therefore by Proposition \ref{idealcase}, we get $\deg h' +\deg u_1-1\geq \v(I)=\sum\limits_{i=1}^r \deg u_i -r$. 
Since $\displaystyle  u\prod\limits_{i=2}^r\frac{u_i}{y_i}= h' u_1^{c_1}\cdots u_r^{c_r}$ with $\sum\limits_{i=1}^r c_i=n-1$ and $\deg u_i\geq \alpha(I)$ for all $1\leq i\leq r$, we get
	\begin{eqnarray*}
	\deg u&=&\deg h' +\sum\limits_{i=1}^r c_i\deg u_i -\sum\limits_{i=2}^r(\deg u_i -1)\\& \geq& \sum\limits_{i=1}^r \deg u_i -r-\deg u_1+1 +(n-1)\alpha(I)-\sum\limits_{i=2}^r\deg u_i +(r-1) \geq (n-1)\alpha(I).\end{eqnarray*}
	
	$(2)$ By Proposition \ref{idealcase} and part $(1)$  of the result, we get \begin{eqnarray*}\v_P(I^n)&\geq& (n-1)\alpha(I)+\deg(\prod\limits_{i=1}^r \frac{u_i}{y_i})=(n-1)\alpha(I)+\sum\limits_{i=1}^r\deg u_i -r\\&=&(n-1)\alpha(I)+\v(I)=n\alpha(I)+\v(I)-\alpha(I).\end{eqnarray*}
	
Let $\Ass(I)=\{P=P_1,\ldots,P_s\}$ and $Q_i$ is the corresponding $P_i$-primary component of $I$. Then 
 $I^{(n)}=I^n=\bigcap\limits_{i=1}^sQ_i^n$ for all $n\geq 1$. 	Let $u_i=x_{i_1}^{\alpha_{i_1}}\cdots x_{i_{t_i}}^{\alpha_{i_{t_i}}}$ for all $1\leq i\leq r$. Without loss of generality, we assume that $y_i=x_{i_1}$ for all $1\leq i\leq r$. Let $\beta_i$ denote $\alpha_{i_1}$ for all $1\leq 1\leq r$. Then $Q_1=\langle x_{i_1}^{\alpha_{i_1}}:1\leq i\leq r\rangle=\langle y_{1}^{\beta_{1}},\ldots,y_{r}^{\beta_{r}}\rangle$ is the corresponding $P=\langle x_{i_1}:1\leq i\leq r\rangle=\langle y_{1},\ldots,y_{r}\rangle$-primary component of $I$. 
\\ Let  $\displaystyle f=\prod\limits_{i=1}^rv_i$ where $v_1=u_1^n/y_{1}$ and $\displaystyle v_m=u_m/y_{m}$ for $2\leq m\leq r$. Then $\displaystyle f=f_1\prod\limits_{j=2}^{t_1}x_{1_j}^{n\alpha_{1_j}}\prod\limits_{m=2}^r\prod\limits_{j=2}^{t_m}x_{m_j}^{\alpha_{m_j}}$ where $f_1=y_1^{n\beta_1-1}\prod\limits_{i=2}^r y_i^{\beta_i-1}=y_1^{(n-1)\beta_1}\prod\limits_{i=1}^r y_i^{\beta_i-1}$. If $|\supp(u_s)|=1$ then we consider $\prod\limits_{j=2}^{t_s}x_{s_j}^{\alpha_{s_j}}=1$ for all $1\leq s\leq r$. 
\\Note that $f_1y_i\in\langle y_1^{(n-1)\beta_1}y_i^{\beta_i}\rangle\in Q_1^n$ for all $1\leq i\leq r$ and hence $P\subseteq (Q_1^n:f_1)$. Since every monomial $u\in\mathcal{G}(Q_1^n)$ is either of the form $u=y_{1}^{n\beta_{1}}$ or $y_i^{\beta_i}$ divides $u$ for some $2\leq i\leq r$, we have $u\nmid f_1$ for any $\displaystyle u\in\mathcal{G}(Q^n)$. Therefore $(Q_1^n:f_1)\neq S$ and $\Ass(S/(Q_1^n:f_1)=\Ass(S/Q_1^n)=\{P\}$.  Hence $P=(Q_1^n:f_1)$. Since $P=\langle y_{1},\ldots,y_{r}\rangle$, we have $$(Q_1^n:f)=((Q_1^n:f_1):\prod\limits_{j=2}^{t_1}x_{1_j}^{n\alpha_{1_j}}\prod\limits_{m=2}^r\prod\limits_{j=2}^{t_m}x_{m_j}^{\alpha_{m_j}})=(P:\prod\limits_{j=2}^{t_1}x_{1_j}^{n\alpha_{1_j}}\prod\limits_{m=2}^r\prod\limits_{j=2}^{t_m}x_{m_j}^{\alpha_{m_j}})=P.$$ 
 \\Now we show that $(Q_j^n:f)=S$ for all $j\neq 1$. 
 Consider a primary component $Q_j$ such that $j\neq 1$. Suppose $y_1^{\beta_1}\in Q_j$. Since $u_2,\ldots,u_r\in Q_j$, there exists $x_{p_i}^{\alpha_{p_i}}\in Q_j\setminus Q$ for some $2\leq i\leq t_p$ and $2\leq p\leq r$. Note that $y_{1}^{(n-1)\beta_{1}}x_{p_i}^{\alpha_{p_i}}\in Q_j^n$ and $y_{1}^{(n-1)\beta_{1}}x_{p_i}^{\alpha_{p_i}}$ divides $f$. Therefore $f\in Q_j^n$. Suppose $y_1^{\beta_1}\notin Q_j$. Then there exists $2\leq i\leq t_1$ such that $x_{1_i}^{\alpha_{1_i}}\in Q_j$. Note that $x_{1_i}^{n\alpha_{1_i}}\in Q_j^n$ and $x_{1_i}^{n\alpha_{1_i}}$ divides $f$. Therefore $f\in Q_j^n$. Thus $(I^n:f)=(Q_1^n:f)\bigcap (\bigcap\limits_{m=2}^r (Q_m^n:f))=P$ and 
$\v_P(I^n)\leq \deg f =n\deg u_1+\sum\limits_{m=2}^r \deg u_m-r=n\alpha(I)+\v(I)-\alpha(I).$
\end{proof}		
\section{$\v$-numbers of integral closures of complete intersection ideals}\label{section4}
In this section, we investigate the $\v$-numbers  associated to integral closure filtrations of complete intersection ideals and analyze their relationship with the Castelnuovo-Mumford regularity of those ideals. We start with the following result which holds for any homogeneous ideals and generalizes the result \cite[Corollary 3.8]{GSV}.
\begin{Lemma}{\label{G}}
	Let $I$ be a homogeneous ideal in $S$. Then $\displaystyle\lim\limits_{n\to\infty}\frac{\alpha(\overline{I^n})}{n}=\alpha(I)$.
	\end{Lemma}	
	\begin{proof}
Note that $\alpha(\overline{I^n})\leq \alpha({I^n})=n\alpha({I})$ for all $n\geq 1$. Now by \cite{MRK} and \cite[Theorem 3.2]{AS}, we have $$\displaystyle\alpha({I})\leq \lim\limits_{n\to\infty}\frac{\v(\overline{I^n})}{n}=\lim\limits_{n\to\infty}\frac{\alpha(\overline{I^n})}{n}\leq\lim\limits_{n\to\infty}\frac{n\alpha({I})}{n}=\alpha(I).$$
		\end{proof}
	\begin{Remark}\label{onecase}{\em
			\begin{enumerate}
				\item[$(1)$] Let $I=(a)$ be a nonzero monomial principal ideal in $S$. Then $I^m=(a^m)$ is integrally closed by \cite[Proposition 1.5.2]{SH} and $\Ass(I)=\MinAss(I)=\MinAss(I^m)=\Ass(I^m)$ for all $m\geq 1$ by \cite[Proposition 4.1.1]{SH}. Consider any $m\geq 1$ and  $P\in\Ass(I^m)$.  Then  by Theorem \ref{ci}, we have $\v_P(\overline{I^m})=\v_P(I^m)=m\deg a -1$ for all $m\geq 1$. Therefore all the ideals, we consider in this section, have height at least two.
				\item[$(2)$] Let $I=P+J$ be a monomial ideal in $S$ where $P$ is a prime monomial ideal and $\mathcal G(P)$, $\mathcal G(J)$ are the sets of monomials with a disjoint set of variables. Then by \cite[Theorem 2.1]{MT}, $\overline I=P+\overline J$. Further using \cite[Theorem 4.1]{FM}, we get $\v_{P+Q}(\overline I)=\v(P)+\v_Q(\overline  J)=\v_Q(\overline  J)$.
				\end{enumerate}
				}
			\end{Remark}
			We first prove a technical lemma which we use in our results
	\begin{Lemma}\label{ceil}
		If $A$, $c$ and $s$ are real numbers with $A\geq 1$, $-1< c\leq 0$ and $0\leq s\leq 1$. Then for all $L\in\mathbb N_{>0}$, the following hold.
		\begin{enumerate}
			\item[$(1)$] $\lceil(L+c)A-s\rceil>\lceil(L-1+c)A-s\rceil$
			\item[$(2)$] $\lceil(L+c)A-s\rceil\geq (L-t)+\lceil(c+t)A-s\rceil$ for all $1\leq t\leq L$.
		\end{enumerate}
	\end{Lemma}
	\begin{proof}
		$(i)$ Since $(L+c)A-s\geq(L-1+c)A-s+1$, we get the required result. 
\\$(ii)$ We use induction on $L$. If $L=1$, then $t=1$ and
		$\lceil(L+c)A-s\rceil=\lceil(1+c)A-s\rceil=(L-1)+\lceil(c+1)A-s\rceil$.
	Suppose $L\geq 2$ and $\lceil(j+c)A-s\rceil\geq (j-t)+\lceil(c+t)A-s\rceil$ for all $1\leq t\leq j$ and $1\leq j\leq L-1$. Then for all $1\leq t\leq j$, using part $(1)$ and the induction hypothesis, we get 
	\begin{eqnarray*}
		\lceil((j+1)+c)A-s\rceil &\geq& 1+ \lceil((j+1)-1+c)A-s\rceil=1+\lceil(j+c)A-s\rceil\\&\geq& 1+(j-t)+\lceil(c+t)A-s\rceil=((j+1)-t)+\lceil(c+t)A-s\rceil.
	\end{eqnarray*} If $t=j+1$ then $\lceil((j+1)+c)A-s\rceil =	((j+1)-(j+1))+\lceil(c+(j+1))A-s\rceil $. Thus we get the required result.
	\end{proof}
	In the next result, for an ideal $I$ generated by pure powers of variables, we analyze the monomials $f\in S$ satisfying $(\overline I:f)=\sqrt{I}$ 	and $\deg f=\v(\overline I)$.
	\begin{Proposition}
		Let $I=\langle x_{i_1}^{a_1},\dots, x_{i_r}^{a_r}\rangle$ be an ideal in $S$ generated by pure powers of variables with $2\leq a_1\leq a_2\leq\cdots\leq a_r$ and $f\in S$ be a monomial such that $(\overline{I}:f)=\langle x_{i_1},\dots, x_{i_r}\rangle$ and $\v(\overline{I})=\deg f$. Then there exists $g\in\mathcal{G}(\overline{I})$ such that $f={g}/{x_{i_r}}$.
	\end{Proposition}
	\begin{proof}
		Let $P=\langle x_{i_1},\dots, x_{i_r}\rangle$.
		Without loss of generality, we assume that $x_{i_j}=x_j$ for all $1\leq j\leq r$. Let $f=\prod\limits_{i=1}^r x_i^{b_i}$. Since  $fx_r\in \overline I$, there exists $g=\prod\limits_{i=1}^r x_i^{c_i}\in\mathcal{G}(\overline{I})$ such that $g\mid fx_r$. We show that $c_i=b_i$ for all $1\leq i\leq r-1$ and $c_r= b_r+1$. Note that  $c_i\leq b_i$ for all $1\leq i\leq r-1$ and $c_r\leq b_r+1$. Suppose $c_j<b_j$ for some $1\leq j\leq r-1 $. Let $f'=\prod\limits_{i=1}^r x_i^{d_i}$ where $d_j=b_j-1$, and $d_i=b_i$ for all $1\leq i\leq r$ and $i\neq j$. Note that $g\mid f' x_r$ implies $f' x_r\in\overline I$. Therefore by Lemma \ref{membership}, we have $\displaystyle\sum\limits_{p=1}^r\frac{d_p}{a_p}+\frac{1}{a_i}\geq \sum\limits_{p=1}^r\frac{d_p}{a_p}+\frac{1}{a_r} \geq 1$ and hence $f' x_i\in\overline{I}$ for all $1\leq i\leq r$. Thus  $P\subset (\overline{I}:f')$. Since $f\notin\overline I$ and $f'$ divides $f$, we have   $f'\notin\overline{I}$. 
	Therefore $ \Ass(R/(\overline{I}:f'))=\Ass(R/\overline{I})=\{P\}$ and $(\overline I:f')=P$. Since $\deg f'<\deg f$, it contradicts that $\v(\overline I)=\deg f$. Hence $c_j=b_j$ for all $j=1,\ldots,r-1$.
		If $c_r<b_r+1$, then $c_r\leq b_r$. Then by Lemma \ref{membership},
		$\displaystyle\sum\limits_{i=1}^r \frac{b_i}{a_i}\geq\sum\limits_{i=1}^r \frac{c_i}{a_i}\geq 1$ and $f\in\overline I$ which contradicts that $(\overline I:f)\neq S$. Hence $c_r=b_r+1$ and $g=fx_r$. 
	\end{proof}
Our next objective is to explore the $\v$-numbers of the integral closures of irreducible monomial ideals.
		\begin{Theorem}\label{gen}
		Let $I=\langle x_{i_1}^{a_1},\ldots,x_{i_k}^{a_k}\rangle$ be a monomial ideal in $S$ such that $1\leq a_1\leq \cdots\leq a_k$ and $k\geq 2$. Then for all $n\geq 1$,
		\begin{enumerate}
			\item[$(1)$]   $\displaystyle n\alpha(I)
			+\lceil \frac{a_2}{\alpha(I)}-\frac{a_2}{\delta(I)}\rceil-1\leq \v(\overline {I^n})\leq n\alpha(I)+\lceil \delta(I)/\alpha(I)\rceil-2$ for all $n\geq 1$.
			\item[$(2)$] Suppose there exists $j\geq2$ such that $a_1=\cdots=a_{j-1}\leq a_j=\cdots=a_k$, then $\displaystyle\v(\overline{I^n})=n\alpha(I)+\lceil\frac{\delta(I)}{\alpha(I)}\rceil-2$.
			\item[$(3)$]  $\v(\overline {I^n})= n\alpha(I)+\lceil \delta(I)/\alpha(I)\rceil-2$ when $k=2$.
			\item[$(4)$]  
		\[\displaystyle 0\leq \reg(S/\overline {I^n})-\v(\overline {I^n})\leq \left\{
			\begin{array}{l l}
				~~~~~\displaystyle (\delta(I)-\alpha(I))n+\dim S/I& \quad \text{if $a_2=\delta(I)$ }\\ \vspace{0.3mm}\\
				\displaystyle (\delta(I)-\alpha(I))n+\dim S/I-1& \quad \text{if $a_2<\delta(I)$ . }\\ 
			\end{array} \right.\] 
			In particular, if $I$ is an equigenerated irreducible monomial ideal then for all $n\geq 1$, $\reg(S/\overline {I^n})=\v(\overline {I^n})=n\alpha(I)-1$. 
		\end{enumerate}
	\end{Theorem}
	\begin{proof} 
		By Remark \ref{powers}, we have $\overline {I^n}=\overline{\langle x_{i_1}^{na_1},\ldots,x_{i_k}^{na_k}\rangle}$ for all $n\geq 1$. Note that $\alpha(I)=a_1$ and $\delta(I)=a_k$. 	Without loss of generality, we assume that $x_{i_j}=x_j$ for all $1\leq j\leq k$. Let $P=\langle x_1,\ldots,x_k\rangle$.
		
		$(1)$ Let $f=x_1^{b_1}\cdots x_k^{b_k}$ be a monomial such that $(\overline{I^n}:f)=P$ and $\v(\overline{I^n})=\deg f$. We show that $\v(\overline{I})\geq na_1+\lceil a_2/a_1-a_2/a_k\rceil-1$. Since $f\notin\overline{I^n}$, we have $b_1=na_1-m$ for some $1\leq m\leq na_1$. 	We have $k\geq 2$. Since $fx_j\in\overline{I^n}$, by Lemma \ref{membership}, we have $\displaystyle \frac{b_1}{na_1}+\sum\limits_{i=2}^k\frac{b_i}{na_i}+\frac{1}{na_j}\geq 1$  for all $2\leq j\leq k$.
		Therefore
		$$\displaystyle\frac{\sum\limits_{i=2}^k b_i}{na_2}\geq \sum\limits_{i=2}^k\frac{b_i}{na_i}\geq 1-\frac{b_1}{na_1}-\frac{1}{na_j}=1-\frac{na_1-m}{na_1}-\frac{1}{na_j}=\frac{m}{na_1}-\frac{1}{na_j}.$$
		 Taking $A=a_2/a_1$, $c=0$, $s=a_2/a_j$, $t=1$ and $L=m$  in Lemma \ref{ceil}, we get
		$$\sum\limits_{i=1}^k b_i\geq b_1+\sum\limits_{i=2}^k b_i\geq na_1-m+\lceil \frac{ma_2}{a_1}-\frac{a_2}{a_j}\rceil\geq na_1-1+\lceil \frac{a_2}{a_1}-\frac{a_2}{a_j}\rceil.$$ Hence $\displaystyle n\alpha(I)-1
		+\lceil \frac{a_2}{\alpha(I)}-\frac{a_2}{\delta(I)}\rceil= n\alpha(I)-1
		+\max\limits_{2\leq j\leq k}\{\lceil \frac{a_2}{\alpha(I)}-\frac{a_2}{a_j}\rceil\}\leq \v(\overline {I^n})$.
		\\Consider the monomial $\displaystyle g=\displaystyle x_1^{na_1-1}x_k^{\lceil\frac{a_k}{a_1}\rceil-1}$. We show that $(\overline{I^n}:g)=P$. Note that $\displaystyle\frac{na_1-1}{na_1}+\frac{1}{na_k}(\lceil\frac{a_k}{a_1}\rceil-1)< \frac{na_1-1}{na_1}+\frac{1}{na_k}(\frac{a_k}{a_1})\leq 1.$ Hence by Lemma \ref{membership}, $g\notin\overline{I^n}$ and $\Ass(R/(\overline{I^n}:g))=\Ass(R/\overline{I^n})=\{P\}$. Now we show that $gx_i\in\overline{I^n}$ for all $i=1,\ldots,k$ and thus $P\subset(\overline{I^n}:g)$. Note that for all $1\leq i\leq k$, 	\begin{eqnarray*}
			\frac{na_1-1}{n a_1}+\frac{1}{na_i}+\frac{1}{na_k}(\lceil\frac{a_k}{a_1}\rceil-1)&\geq& 1-\frac{1}{na_1}+\frac{1}{na_i}+\frac{1}{na_k}(\frac{a_k}{a_1}-1)\\&=&1+\frac{1}{na_i}-\frac{1}{na_k}\geq 1.\end{eqnarray*} 
			Thus by Lemma \ref{membership}, we have $gx_i\in\overline{I^n}$ for all $1\leq i\leq k$. Hence $(\overline{I^n}:g)=P$ and $\displaystyle\v(\overline{I^n})\leq \deg g= na_1+\lceil\frac{a_k}{a_1}\rceil-2$.
		\\$(2)$ By part $(1)$, we have $\v(\overline {I^n})\leq n\alpha(I)+\lceil \delta(I)/\alpha(I)\rceil-2$. Let $f=x_1^{b_1}\cdots x_r^{b_r}$ be a monomial such that $(\overline{I^n}:f)=P$ and $\v(\overline{I^n})=\deg f$.    
		If $\sum\limits_{i=1}^{j-1}b_i\geq na_1$, then $\displaystyle\frac{\sum\limits_{i=1}^{j-1}b_i}{na_1}\geq 1$ and hence by Lemma \ref{membership}, $f\in \overline{I^n}$, which is a contradiction. Therefore $\sum\limits_{i=1}^{j-1}b_i= na_1-m$ for some $1\leq m\leq na_1$. 
		Since $fx_j\in\overline{I^n}$, by Lemma \ref{membership}, we have $\displaystyle\frac{\sum\limits_{i=1}^{j-1}b_i}{na_1}+\frac{\sum\limits_{i=j}^{r}b_i+1}{na_j}\geq 1$ and hence
		$
		\displaystyle\frac{\sum\limits_{i=j}^{r}b_i+1}{na_j}\geq  1-\frac{(na_1-m)}{na_1}=\frac{m}{na_1}.
		$
		Therefore $\displaystyle{\sum\limits_{i=j}^{r}b_i+1}\geq\lceil\frac{ma_j}{a_1}\rceil$. Thus taking $A=a_j/a_1$, $c=0$, $s=0$, $t=1$ and $L=m$  in Lemma \ref{ceil}, we have 
		$ \displaystyle \sum\limits_{i=1}^r b_i \geq na_1-m+ \lceil\frac{ma_j}{a_1}\rceil-1  \geq na_1+\lceil\frac{a_j}{a_1}\rceil-2=na_1+\lceil\frac{\delta(I)}{\alpha(I)}\rceil-2.$
		\\$(3)$ It follows from $(2)$.
		\\$(4)$ By \cite[Theorem 2.7]{Hoa} and \cite[Lemma 7.3]{GSV}, we have $\delta(I)n\leq\displaystyle{\reg({\overline{I^n}})}\leq \delta(I)n+\dim S/I$. Let $\delta(I)=s\alpha(I)+t$ for some $0\leq t<\alpha(I)$. For all $n\geq 1$, using part $(1)$, we get 
		\begin{eqnarray*}\displaystyle\reg(S/\overline {I^n})-\v(\overline {I^n})&\geq& \delta(I)n-1-(n\alpha(I)+\lceil \delta(I)/\alpha(I)\rceil-2)\\&=&(\delta(I)-\alpha(I))n-\lceil \delta(I)/\alpha(I)\rceil+1\\&\geq &(\delta(I)-\alpha(I))-\lceil \delta(I)/\alpha(I)\rceil+1\\&=& (s-1)\alpha(I)+t-\lceil\frac{s\alpha(I)+t}{\alpha(I)}\rceil+1\\&=& (s-1)(\alpha(I)-1)+t-\lceil\frac{t}{\alpha(I)}\rceil\geq 0. \end{eqnarray*}
	Now for all $n\geq 1$, using part $(1)$, we get
	\begin{eqnarray*}\displaystyle\reg(S/\overline {I^n})-\v(\overline {I^n})&\leq& \delta(I)n+\dim S/I-1-(n\alpha(I)
	+\lceil \frac{a_2}{\alpha(I)}-\frac{a_2}{\delta(I)}\rceil-1)
	\\&=& (\delta(I)-\alpha(I))n+\dim S/I-\lceil \frac{a_2}{\alpha(I)}-\frac{a_2}{\delta(I)}\rceil.
	\end{eqnarray*}
	Note that $\displaystyle\lceil\frac{a_2}{\alpha(I)}-\frac{a_2}{\delta(I)}\rceil\geq \lceil 1-\frac{a_2}{\delta(I)}\rceil=1+\lceil -\frac{a_2}{\delta(I)}\rceil=1-\lfloor \frac{a_2}{\delta(I)}\rfloor$. If $a_2=\delta(I)$ then $\displaystyle\lfloor \frac{a_2}{\delta(I)}\rfloor=1$. If $a_2<\delta(I)$ then $\displaystyle\lfloor \frac{a_2}{\delta(I)}\rfloor=0$. Therefore we get the required result.
	\\Suppose $I$ is an equigenerated irreducible monomial ideal. Then by Remark \ref{powers}, $\overline {I^n}=\overline{\langle x_{1}^{na},\ldots,x_{k}^{na}\rangle}=\overline{\langle x_{1},\ldots,x_{k}\rangle^{na}}=\overline{P^{na}}$ where $a=a_1=\ldots=a_k$. Now by \cite[Proposition 1.4.4 and Theorem 1.4.6]{HH}, we have $\overline {I^n}=P^{na}$. Therefore by \cite[Lemma 4.4]{BHT} and  part $(2)$,  $\reg(S/\overline {I^n})=\reg(S/P^{na})=na-1=n\alpha(I)-1=\v(\overline {I^n})$ for all $n\geq 1$.
	\end{proof}
	Motivated by the above result, we compute the $\v$-numbers associated to the integral closure filtrations of the edge ideal of a weighted oriented graph whose underlying graph is a complete bipartite graph. For more details about weighted oriented graphs, see \cite{GMBV}. 
	\begin{Corollary}\label{WOG}
		Let $K_{p_1,p_2}$ be a complete bipartite graph with the vertex set $V=V_1\sqcup V_2$ where $V_1=\{x_1,\dots,x_{p_1}\}$ and $V_2=\{y_1,\dots,y_{p_2}\}$. Let $D=((V(D),E(D),w)$ be a weighted oriented graph whose underlying graph is $K_{p_1,p_2}$ such that $E(D)=\{(x_i,y_j)\mid 1\le i\le p_1,1\le j\le p_2\}$. Then 
		\begin{enumerate} 
		\item[$(1)$] $\v(\overline{I(D)^n})=n\alpha(I(D))-1$ for all $n\geq 1$. 
		\item[$(2)$]  If $K_{p_1,p_2}=K_{1,p}$ with $p\geq 3$ and  $1\leq w(y_i)\leq p-2$ for some $1\leq i\leq p$ then $$\v(\overline{I(D)})<\dim R/(\overline{I(D)})-\depth R/(\overline{I(D)})$$ where $R=K[x_1,y_1,\dots,y_{p}]$. In particular, if
	 $w(y_i)=1$ for some $1\leq i\leq p$ then $\v(\overline{I(D)})=1$ and  $\Big(\dim R/(\overline{I(D)})-\depth R/(\overline{I(D)})\Big)-\v(\overline{I(D)})= p-2$.
		\end{enumerate}
		
	\end{Corollary}	
	\begin{proof}
	Let $\beta_j=w(y_j)$ for all $1\le j\le p_2$. Note that the edge ideal of $D$ is $$I(D)=\langle x_iy_j^{\beta_j}\mid 1\le i\le p_1,1\le j\le p_2\rangle=\langle x_1,\dots,x_{p_2}\rangle \cap\langle y_1^{\beta_1},\dots,y_{p_2}^{\beta_{p_2}}\rangle.$$ Let $P=\langle x_1,\dots,x_{p_1}\rangle$, $J=\langle y_1^{\beta_1},\dots,y_{p_2}^{\beta_{p_2}}\rangle$ and $\sqrt{J}=Q=\langle y_1,\dots,y_{p_2}\rangle$. Since $P$ and $J$ are generated by monomials with a disjoint set of variables, for all $n\ge 1$, we have 
		$$I(D)^n=(P\cap J)^n=(PJ)^n=P^nJ^n=P^n\cap J^n=I(D)^{(n)}$$  where the last equality follows from \cite[Theorem 3.7]{CEHH}. 
		\\$(1)$ By \cite[Theorem 4.1]{GMBV}, we have $\overline{I(D)^n}=\overline{P^n}\cap\overline{J^n}=\overline{P^n}\overline{J^n}$ for all $n\geq 1$. Therefore by \cite[Theorem 2.2]{FM}, $\v(\overline{I(D)^n})=\min\{\v_P(\overline{P^n})+\alpha(\overline {J^n}),\v_Q(\overline{J^n})+\alpha(\overline {P^n})\}$ for all $n\geq 1$. Since by Remark \ref{powers},  $\overline{P^n}=\overline{\langle x_1^n,\dots,x_{p_2}^n\rangle}$ and $\overline{J^n}=\overline{\langle y_1^{n\beta_1},\dots, y_{p_2}^{n\beta_{p_2}}\rangle}$, we have $\alpha(\overline {P^n})=n$ and $\alpha(\overline {J^n})=n\alpha(J)$. Now by Theorem \ref{gen}, we have $\v_P(\overline{P^n})=n-1$ and $\v_P(\overline{J^n})\geq n\alpha(J)-1$. Hence $\v(\overline{I(D)^n})=n\alpha(J)+n-1=n(\alpha(J)+1)-1=n\alpha(I(D))-1$.
		\\$(2)$ Note that the underlying graph is $K_{1,p}$ for some $p\geq 3$ and $1\leq w(y_i)\leq p-2$ for some $1\leq i\leq p$. Now $\dim R/\overline{I(D)}=p$ and 
		$\depth R/\overline{I(D)}=1$. Therefore $$\v(\overline{I(D)})=\alpha(I(D))-1\leq (w(y_i)+1)-1\leq (p-1)-1<\dim R/\overline{I(D)}-\depth R/\overline{I(D)}.$$ 
		For the last part, we have $\alpha(I(D))=2$, $\v(\overline{I(D)})=\alpha(I(D))-1=1$ and we get the required result.
		 
		\end{proof}	
		Recently, a negative answer to \cite[Question 5.5]{SS} was obtained in \cite[Example 8.6]{KMT} using an ideal with a linear resolution. Based on the above result, we provide a negative answer to the same question by constructing ideals which are not necessarily equigenerated (and thus may not admit linear resolutions).
		\begin{Corollary}\label{negativeanser}
			For any integer $p\geq 3$, there exists a squarefree monomial ideal $I$ (not necessarily equigenerated) in a polynomial ring $R$ such that $$\big(\dim R/I-\depth R/I\big)-\v(I)=p-2.$$
		\end{Corollary}	
		\begin{proof}
Let $G$ denote a complete bipartite graph $K_{1,p}$ with $p\geq 3$. Consider the weighted oriented graph $D=((V(D),E(D),w)$ whose underlying graph is $G$ such that $E(D)=\{(x,y_i)\mid 1\le j\le p\}$ and $w(y_1)=1$. Let $J=\overline{I(D)}$ and $I=J^{\pol}$ (i.e., $I$ is the polarization of $J$). Let $T=K[x,y_1,\ldots,y_p]$ and $R$ be the polynomial ring obtaining from $T$ adjoining variables needed for polarization of $J$.
\\By \cite[Theorem 4.1]{F23} and the Corollary \ref{WOG}, we have $\v(I)=\v(J)=\v(\overline{I(D)})=1$. Since $\dim R/I-\depth R/I=\dim T/J-\depth T/J$ by \cite[Corollary 1.6.3]{HH}, we get the required result using Corollary \ref{WOG}
			\end{proof}
Based on Theorem \ref{gen}, we focus on explicit computation of $\v$-numbers of the integral closure filtrations of $\height 3$ irreducible monomial ideals. We begin by proving the following result.
	\begin{Proposition}\label{monomialofvint}
		Let $I=\langle x_{i_1}^{a_1},x_{i_2}^{a_2},x_{i_3}^{a_3}\rangle$ be a monomial ideal in $S$ such that $1\leq a_1\leq a_2\leq a_3$. For all $1\leq m\leq a_1$, consider the monomials, $$\displaystyle {\huge {f_{m}=x_{i_1}^{a_1-m}x_{i_2}^{\lceil\frac{ma_2}{a_1}\rceil-1}x_{i_3}^{\lceil\frac{a_3}{a_2}(\frac{ma_2}{a_1}-\lceil\frac{ma_2}{a_1}\rceil+1)\rceil-1}\in S}}.$$ Then $(\overline{I}:f_m)=\langle x_{i_1},x_{i_2},x_{i_3}\rangle$ for all $1\leq m\leq a_1$. 
	\end{Proposition}
	\begin{proof}
		Let $P=\langle x_{i_1},x_{i_2},x_{i_3}\rangle$. Without loss of generality, assume that $x_{i_j}=x_j$ for all $1\leq j\leq 3$.  We first show that for all $1\leq i\leq 3$ and $1\leq m\leq a_1$, $f_m x_i\in\overline{I}$. 
		\\Note that $\displaystyle0<\frac{ma_2}{a_1}-\lceil\frac{ma_2}{a_1}\rceil+1$ and  $\displaystyle 1\leq \lceil\frac{a_3}{a_2}(\frac{ma_2}{a_1}-\lceil\frac{ma_2}{a_1}\rceil+1)\rceil$. Thus for $i=1,2$,   \begin{eqnarray*}
			&& \frac{a_1-m}{a_1}+\frac{1}{a_2}(\lceil\frac{ma_2}{a_1}\rceil-1)+\frac{1}{a_i}+\frac{1}{a_3}(\lceil\frac{a_3}{a_2}(\frac{ma_2}{a_1}-\lceil\frac{ma_2}{a_1}\rceil+1)\rceil-1)\\ &\geq& 1-\frac{m}{a_1}+\frac{1}{a_2}(\frac{ma_2}{a_1})-\frac{1}{a_2}+\frac{1}{a_i}= 1-\frac{1}{a_2}+\frac{1}{a_i}\geq 1
		\end{eqnarray*} and for $i=3$, 
		\begin{eqnarray*}
			&& \frac{a_1-m}{a_1}+\frac{1}{a_2}(\lceil\frac{ma_2}{a_1}\rceil-1)+\frac{1}{a_3}(\lceil\frac{a_3}{a_2}(\frac{ma_2}{a_1}-\lceil\frac{ma_2}{a_1}\rceil+1)\rceil-1)+\frac{1}{a_3}\\ &\geq& 1-\frac{m}{a_1}+\frac{1}{a_2}(\lceil\frac{ma_2}{a_1}\rceil-1)+\frac{1}{a_3}(\frac{a_3}{a_2}(\frac{ma_2}{a_1}-\lceil\frac{ma_2}{a_1}\rceil+1))-\frac{1}{a_3}+\frac{1}{a_3}\\&=& 1-\frac{1}{a_2}(\frac{ma_2}{a_1}-\lceil\frac{ma_2}{a_1}\rceil+1)+\frac{1}{a_2}(\frac{ma_2}{a_1}-\lceil\frac{ma_2}{a_1}\rceil+1)=1. 
		\end{eqnarray*}
	Therefore by Lemma \ref{membership}, we have $P\subseteq (\overline{I}:f_m)$ for all $1\leq m\leq a_1$.
	\\Now we show that $f_m\notin \overline{I}$ for all $1\leq m\leq a_1$. Since for any real number $x$, we have $\lceil{x}\rceil-1<x$, we get \begin{eqnarray*}
			&& \frac{a_1-m}{a_1}+\frac{1}{a_2}(\lceil\frac{ma_2}{a_1}\rceil-1)+\frac{1}{a_3}(\lceil\frac{a_3}{a_2}(\frac{ma_2}{a_1}-\lceil\frac{ma_2}{a_1}\rceil+1)\rceil-1)\\ &<& 1-\frac{m}{a_1}+\frac{1}{a_2}(\lceil\frac{ma_2}{a_1}\rceil-1)+\frac{1}{a_3}(\frac{a_3}{a_2}(\frac{ma_2}{a_1}-\lceil\frac{ma_2}{a_1}\rceil+1))\\&=& 1-\frac{1}{a_2}(\frac{ma_2}{a_1}-\lceil\frac{ma_2}{a_1}\rceil+1)+\frac{1}{a_2}(\frac{ma_2}{a_1}-\lceil\frac{ma_2}{a_1}\rceil+1)=1.
		\end{eqnarray*} 
		By Lemma \ref{membership}, we have $f_m\notin\overline{I}$ for all $1\leq m\leq a_1$. Thus $\Ass(R/(\overline{I}:f_m))=\Ass(R/\overline{I})=\{P\}$ and hence $(\overline{I}:f_m)=P$ for all $1\leq m\leq a_1$.
	\end{proof}
	\begin{Theorem}\label{threegen}
		Let $I=\langle x_{i_1}^{a_1},x_{i_2}^{a_2},x_{i_3}^{a_3}\rangle$ be a monomial ideal in $S$ with $1\leq a_1\leq a_2\leq a_3$. 
		\begin{enumerate}
			\item[$(1)$] $\v(\overline{I})=\min\limits_{1\leq m\leq a_1} \deg f_m$.
			\item[$(2)$] If $a_1=1$ then $\v(\overline{I^n})=n+a_2+\lceil\frac{a_3}{a_2}\rceil-3$ for all $n\geq 1$. 
	\item[$(3)$] Let $a_1\geq 2$.
	\begin{enumerate}
		\item[$(i)$] $\displaystyle\v(\overline{I})=\min\limits_{1\leq m\leq a_1-1} \deg f_m$.
		\item[$(ii)$] Let $\v(\overline{I})=\deg f_l$ for some $1\leq l\leq a_1-1$. Then for all $n\geq 1$, $$\displaystyle\v(\overline{I^n})=(n-1)a_1+\deg f_l=na_1-l+\lceil\frac{la_2}{a_1}\rceil+\lceil\frac{a_3}{a_2}(\frac{l a_2}{a_1}-\lceil\frac{la_2}{a_1}\rceil+1)\rceil-2.$$
		\item[$(iii)$] If $a_2\equiv0~(\mod a_1)$ or $a_2\equiv1~(\mod a_1)$, then $\v(\overline{I})=\deg f_1$ and for all $n\geq 1$,
		$$\displaystyle\v(\overline{I^n})=(n-1)a_1+\deg f_1=na_1+\lceil\frac{a_2}{a_1}\rceil+\lceil\frac{a_3}{a_2}(\frac{ a_2}{a_1}-\lceil\frac{a_2}{a_1}\rceil+1)\rceil-3.$$
	\item[$(iv)$] $\displaystyle\lceil a_2/a_1\rceil-1\leq \lceil a_2/a_1-a_2/a_3\rceil\leq \lceil a_2/a_1\rceil$ and if $\displaystyle\lceil a_2/a_1\rceil-1=\lceil a_2/a_1-a_2/a_3\rceil$ then for all $n\geq 1$, $\displaystyle\v(\overline{I^n})=na_1+\lceil a_2/a_1\rceil-2.$  
	\item[$(v)$] If $\displaystyle\lceil\frac{(a_1-1)a_3}{a_1a_2}\rceil=1$ then $\displaystyle\v(\overline{I^n})=(n-1)a_1+\deg f_1$ for all $n\geq 1$.
\end{enumerate}	
	\end{enumerate}
	Here $f_m$ are the monomials in Proposition \ref{monomialofvint}.
	\end{Theorem}
	\begin{proof}
		Without loss of generality, we assume that $x_{i_j}=x_j$ for all $1\leq j\leq 3$. Let $P=\langle x_1,x_2,x_3\rangle$. By Remark \ref{powers}, for all $n\geq 1$,  $\overline{I^n}=\overline{\langle{x_1^{na_1},x_2^{na_2},x_3^{na_3}}\rangle}$.
\\$(1)$ Consider the monomials $$\displaystyle f_{m}=x_1^{a_1-m}x_2^{\lceil\frac{ma_2}{a_1}\rceil-1}x_3^{\lceil\frac{a_3}{a_2}(\frac{ma_2}{a_1}-\lceil\frac{ma_2}{a_1}\rceil+1)\rceil-1}\in S$$ for all $1\leq m\leq a_1$. By Proposition \ref{monomialofvint}, we know that $(\overline{I}:f_m)=P$ for all $1\leq m\leq a_1$.  Hence $\v(\overline{I})\leq \min\limits_{1\leq m\leq a_1} \deg f_m$. 
\\Let $f=x_1^{b_1} x_2^{b_2} x_3^{b_3}$ be a monomial such that $(\overline{I}:f)=P$ and $\v(\overline{I})=\deg f$. We show that $\deg f\geq \deg f_{m'}$ for some $1\leq m'\leq a_1$. 
\\Since $f\notin\overline{I}$, we have $b_1=a_1-m'$ for some $1\leq m'\leq a_1$. If $\displaystyle b_2\geq \lceil\frac{m'a_2}{a_1}\rceil $ then $$\displaystyle\frac{a_1-m'}{a_1}+\frac{1}{a_2}\lceil\frac{m'a_2}{a_1}\rceil\geq 1-\frac{m'}{a_1}+\frac{m'a_2}{a_2a_1}=1.$$ By Lemma \ref{membership}, we get $f\in\overline{I}$ which is a contradiction. Hence $\displaystyle b_2=\lceil\frac{m'a_2}{a_1}\rceil-l$ for some $\displaystyle 1\leq l\leq \lceil\frac{m'a_2}{a_1}\rceil$. Since $fx_3\in\overline{I}$, again using Lemma \ref{membership}, we have $\displaystyle \frac{b_1}{a_1}+\frac{b_2}{a_2}+\frac{b_3+1}{a_3}\geq 1$. Thus
\begin{eqnarray*}
	\displaystyle\frac{b_3+1}{a_3}&\geq& 1-\frac{b_1}{a_1}-\frac{b_2}{a_2}\geq 1-\frac{a_1-m'}{a_1}-\frac{1}{a_2}(\lceil\frac{m'a_2}{a_1}\rceil-l) \\ &= & \frac{m'}{a_1}-\frac{1}{a_2}(\lceil\frac{m'a_2}{a_1}\rceil-l)= \frac{1}{a_2}(\frac{m'a_2}{a_1}-\lceil\frac{m'a_2}{a_1}\rceil+l).
\end{eqnarray*}
Therefore $b_3\geq \displaystyle\lceil\frac{a_3}{a_2}(\frac{m'a_2}{a_1}-\lceil\frac{m'a_2}{a_1}\rceil+l)\rceil-1$. Taking $L=l$, $t=1$ $A=a_3/a_2$, $\displaystyle c=\frac{m'a_2}{a_1}-\lceil\frac{m'a_2}{a_1}\rceil$ and $s=0$ in  Lemma \ref{ceil}, we get
\begin{eqnarray*}
	\deg f=b_1+b_2+b_3&\geq& a_1-m'+\lceil\frac{m'a_2}{a_1}\rceil-l+\lceil\frac{a_3}{a_2}(\frac{m'a_2}{a_1}-\lceil\frac{m'a_2}{a_1}\rceil+l)\rceil-1\\&\geq& a_1-m'+\lceil\frac{m'a_2}{a_1}\rceil-l+ l-1+\lceil\frac{a_3}{a_2}(\frac{m'a_2}{a_1}-\lceil\frac{m'a_2}{a_1}\rceil+1)\rceil-1 \\&=&a_1-m'+\lceil\frac{m'a_2}{a_1}\rceil+\lceil\frac{a_3}{a_2}(\frac{m'a_2}{a_1}-\lceil\frac{m'a_2}{a_1}\rceil+1)\rceil-2=\deg f_{m'}.
\end{eqnarray*}
Hence $ \min\limits_{1\leq m\leq a_1}\deg f_m\leq \deg f_{m'}\leq \deg f=\v(\overline{I})\leq \min\limits_{1\leq m\leq a_1}\deg f_m$.
	\\$(2)$ By Remark \ref{powers}, we have $\overline {I^n}=\overline{\langle x_{1}^{n},x_{2}^{na_2},x_{3}^{na_3}\rangle}$ for all $n\geq 1$. By Proposition \ref{monomialofvint} and part $(1)$, we get $\v(\overline{I^n})=\min\limits_{1\leq m\leq n}\deg g_m$ where $\displaystyle g_m=x_{1}^{n-m}x_{2}^{ma_2-1}x_{3}^{\lceil\frac{a_3}{a_2}\rceil-1}.$ Therefore for all $n\geq 1$, we have
		\begin{eqnarray*}
		\v(\overline{I^n})&=&\min\limits_{1\leq m\leq n}\{n-m+ma_2-1+\lceil\frac{a_3}{a_2}\rceil-1\}=\min\limits_{1\leq m\leq n}\{n+m(a_2-1)+\lceil\frac{a_3}{a_2}\rceil-2\}\\&=&n+(a_2-1)+\lceil\frac{a_3}{a_2}\rceil-2=n+a_2+\lceil\frac{a_3}{a_2}\rceil-3.
		\end{eqnarray*}	
	$(3)$ $(i)$ Let $a_2\geq 2$. By part (1), we have $\v(\overline{I^n})=\min\limits_{1\leq m\leq a_1}\deg f_m$ where $f_m$ are monomials in Proposition \ref{monomialofvint} for all $1\leq m\leq a_1$. Now we show $\deg f_{a_1}\geq \deg f_1$. Note that 
	\begin{eqnarray}\label{equfor1}
		\deg f_{a_1}-\deg f_1&=&a_2+\lceil\frac{a_3}{a_2}\rceil-2-(a_1+\lceil\frac{a_2}{a_1}\rceil+\lceil\frac{a_3}{a_2}(\frac{a_2}{a_1}-\lceil\frac{a_2}{a_1}\rceil+1)\rceil-3)\nonumber\\&=&a_2-(a_1+\lceil\frac{a_2}{a_1}\rceil-1)+ \lceil\frac{a_3}{a_2}\rceil-\lceil\frac{a_3}{a_2}(\frac{a_2}{a_1}-\lceil\frac{a_2}{a_1}\rceil+1)\rceil.
	\end{eqnarray}
	If $a_1$ divides $a_2$, i.e., $a_2=ta_1$ for some $t\in\mathbb Z_{>0}$ then by equation (\ref{equfor1}), we have 
	$$\deg f_{a_1}-\deg f_1=a_2-(a_1+\frac{a_2}{a_1}-1)\geq ta_1-a_1-t+1=(t-1)(a_1-1)\geq 0.$$
	Suppose $a_1$ does not divide $a_2$ and  $a_2=s a_1+k$ for some $s\geq 1$ and $1\leq k<a_1$. Then $\displaystyle\lceil\frac{a_2}{a_1}\rceil=s+1$ and $\displaystyle\frac{a_2}{a_1}-\lceil\frac{a_2}{a_1}\rceil+1=k/a_1$. Since $\displaystyle 0<\frac{ka_3}{a_1a_2}<\frac{a_3}{a_2}$, by equation (\ref{equfor1}), we have 
	\begin{eqnarray*}
		\deg f_{a_1}-\deg f_1&=& sa_1+k-(a_1+s)+\lceil\frac{a_3}{a_2}\rceil- \lceil \frac{ka_3}{a_1a_2}\rceil\geq sa_1+k-a_1-s\\&=& s(a_1-1)+k-a_1\geq  a_1-1+k-a_1\geq 0.
	\end{eqnarray*}
	Hence by part $(1)$, $\displaystyle\v(\overline{I})=\min\limits_{1\leq m\leq a_1-1} \deg f_m$. 
\\$(ii)$ By Remark \ref{powers}, we have $\overline {I^n}=\overline{\langle x_{1}^{na_1},x_{2}^{na_2},x_{3}^{na_3}\rangle}$ for all $n\geq 1$. By part $(3)(i)$, we have  
$\displaystyle\v(\overline I)=\min\limits_{1\leq m\leq a_1-1}\deg f_m=\deg f_l$ (where $f_m$ are monomials as mentioned in Proposition \ref{monomialofvint}).
\\Fix $n\geq 2$. For all $1\leq m\leq na_1-1,$ let $$\displaystyle g_m=x_1^{na_1-m}x_2^{\lceil\frac{ma_2}{a_1}\rceil-1}x_3^{\lceil\frac{a_3}{a_1}(\frac{ma_2}{a_1}-\lceil\frac{ma_2}{a_1}\rceil+1)\rceil-1}\in S.$$ By part $(3)(i)$, we have $\displaystyle\v(\overline {I^n})=\min\limits_{1\leq m\leq na_1-1}\deg g_m$. Note that $g_m=x_1^{(n-1)a_1}f_m$ for all $1\leq m\leq a_1-1$. Therefore $$\deg g_m=(n-1)a_1+\deg f_m\geq (n-1)a_1+\deg f_l=\deg g_l.$$
We show that $\deg g_m\geq\deg g_l$ for all $a_1\leq m\leq na_1-1$ and hence $\v(\overline{I^n})=\deg g_l$. Suppose $m\geq a_1$. Let $m=sa_1+r$ for some $s\geq 1$ and $0\leq r\le a_1-1$. 
	\\ {\bf{Case 1:}} Suppose $1\leq r\le a_1-1$ (hence $a_1$ does not divide $m$). Then 
	\begin{eqnarray*}
		\deg g_m &=& na_1-m+\lceil\frac{ma_2}{a_1}\rceil+\lceil\frac{a_3}{a_2}(\frac{ma_2}{a_1}-\lceil\frac{ma_2}{a_1}\rceil+1)\rceil-2 \\ &=& na_1-m+\lceil sa_2+\frac{ra_2}{a_1}\rceil+\lceil\frac{a_3}{a_2}(sa_2+\frac{ra_2}{a_1}-sa_2-\lceil\frac{ra_2}{a_1}\rceil+1)\rceil-2\\ &=& na_1-m+sa_2+\lceil\frac{ra_2}{a_1}\rceil+\lceil\frac{a_3}{a_2}(\frac{ra_2}{a_1}+\lceil\frac{ra_2}{a_1}\rceil+1)\rceil-2\\ &\geq& na_1-m+sa_1+\lceil\frac{ra_2}{a_1}\rceil+\lceil\frac{a_3}{a_2}(\frac{ra_2}{a_1}+\lceil\frac{ra_2}{a_1}\rceil+1)\rceil-2\\ &=& na_1-r+\lceil\frac{ra_2}{a_1}\rceil+\lceil\frac{a_3}{a_2}(\frac{ra_2}{a_1}+\lceil\frac{ra_2}{a_1}\rceil+1)\rceil-2=\deg g_r\geq \deg g_l.
\end{eqnarray*} 
	\\ {\bf{Case 2:}} Suppose $r=0$. Note that $l<a_1\leq m$. Then, since $\displaystyle 0<\frac{la_2}{a_1}-\lceil\frac{la_2}{a_1}\rceil+1\leq 1$, we have \begin{eqnarray*}
	\displaystyle	\deg g_m &=& na_1-m+\lceil\frac{ma_2}{a_1}\rceil+\lceil\frac{a_3}{a_2}(\frac{ma_2}{a_1}-\lceil\frac{ma_2}{a_1}\rceil+1)\rceil-2 \\ &=& na_1-m+\lceil\frac{ma_2}{a_1}\rceil+\lceil\frac{a_3}{a_2}\rceil-2 \\ &\geq& na_1-m+\lceil\frac{ma_2}{a_1}\rceil+\lceil\frac{a_3}{a_2}(\frac{la_2}{a_1}-\lceil\frac{la_2}{a_1}\rceil+1)\rceil-2 \\ &\geq& na_1-m+(m-l)+\lceil\frac{la_2}{a_1}\rceil+\lceil\frac{a_3}{a_2}(\frac{la_2}{a_1}-\lceil\frac{la_2}{a_1}\rceil+1)\rceil-2 = \deg g_l
	\end{eqnarray*}
	where the second last step holds by Lemma \ref{ceil} (taking $L=m$, $t=l$, $A=a_2/a_1$ and $c=s=0$). Therefore $\deg g_m\geq \deg g_l$ for all $1\leq m\leq na_1-1$ and $\displaystyle\v(\overline{I^n})=na_1-l+\lceil\frac{la_2}{a_1}\rceil+\lceil\frac{a_3}{a_2}(\frac{la_2}{a_1}-\lceil\frac{la_2}{a_1}\rceil+1)\rceil-2$ for all $n\geq 1$. 
	\\$(iii)$ We first show that $\v(\overline I)=\deg f_1$. Then the result follows from part $(3)(ii)$.
	\\Let $f_m$ denote the monomials in Proposition \ref{monomialofvint} for all $1\leq m\leq a_1-1.$ By part $(3)(ii)$, we have $\v(\overline{I})=\min\limits_{1\leq m\leq a_1-1} \deg f_m$. 
	\\Let $a_2=sa_1+k$ for some $s\geq 1$ and $0\le k\le a_1-1$. Then for all $1\leq m\leq a_1-1$, taking $L=m$, $A=a_2/a_1$, $t=1$ and $c=s=0$ in Lemma \ref{ceil}, we get,
	\begin{eqnarray}
		\deg f_m&=&a_1-m+\lceil\frac{ma_2}{a_1}\rceil+\lceil\frac{a_3}{a_2}(\frac{ma_2}{a_1}-\lceil\frac{ma_2}{a_1}\rceil+1)\rceil-2\nonumber \\ &\geq& a_1-m +(m-1)+\lceil\frac{a_2}{a_1}\rceil+\lceil\frac{a_3}{a_2}(ms+\frac{km}{a_1}-\lceil ms+\frac{km}{a_1}\rceil+1)\rceil-2\nonumber\\ &=& a_1+\lceil\frac{a_2}{a_1}\rceil+\lceil\frac{a_3}{a_2}(\frac{km}{a_1}-\lceil\frac{km}{a_1}\rceil+1)\rceil-3 \label{eq1}. 
	\end{eqnarray}
	If $k=0$, i.e., $a_1$ divides $a_2$ then using equation (\ref{eq1}), we get, \begin{eqnarray*}
		\deg f_m&=&\displaystyle a_1-m+\lceil\frac{ma_2}{a_1}\rceil+\lceil\frac{a_3}{a_2}(\frac{ma_2}{a_1}-\lceil\frac{ma_2}{a_1}\rceil+1)\rceil-2 \geq  a_1+\lceil\frac{a_2}{a_1}\rceil+\lceil\frac{a_3}{a_2}\rceil-3\\&=&a_1+\lceil\frac{a_2}{a_1}\rceil+\lceil\frac{a_3}{a_2}(\frac{a_2}{a_1}-\lceil\frac{a_2}{a_1}\rceil+1)\rceil-3=\deg f_1.\end{eqnarray*}
	Suppose $k=1$. Since $\displaystyle \lceil\frac{k}{a_1}\rceil=\lceil\frac{m}{a_1}\rceil=1$, using equation (\ref{eq1}), we have \begin{eqnarray*}
		\deg f_m&\geq&a_1+\lceil\frac{a_2}{a_1}\rceil+\lceil\frac{a_3}{a_2}(\frac{km}{a_1}-\lceil\frac{km}{a_1}\rceil+1)\rceil-3 \\&\geq & a_1+\lceil\frac{a_2}{a_1}\rceil+\lceil\frac{a_3}{a_2}(\frac{m}{a_1}-\lceil\frac{m}{a_1}\rceil+1)\rceil-3 = a_1+\lceil\frac{a_2}{a_1}\rceil+\lceil\frac{a_3}{a_2}(\frac{m}{a_1})\rceil-3 \\&\geq & a_1+\lceil\frac{a_2}{a_1}\rceil+\lceil\frac{a_3}{a_2}(\frac{1}{a_1})\rceil-3=  a_1+\lceil\frac{a_2}{a_1}\rceil+\lceil\frac{a_3}{a_2}(\frac{k}{a_1})\rceil-3\\&= & a_1+\lceil\frac{a_2}{a_1}\rceil+\lceil\frac{a_3}{a_2}(\frac{a_2}{a_1}-\lceil\frac{a_2}{a_1}\rceil+1)\rceil-3=\deg f_1.
	\end{eqnarray*}
		$(iv)$ Since $a_2/a_1-a_2/a_2\leq a_2/a_1-a_2/a_3\leq a_2/a_1$ we get the first part of the result.
	 \\	Suppose $ \lceil a_2/a_1\rceil-1=\lceil a_2/a_1-a_2/a_3\rceil$ and let $\lceil a_2/a_1-a_2/a_3\rceil=a_2/a_1-a_2/a_3+t$  for some $0\leq t<1$. Then $$0\leq \lceil a_3/a_2(a_2/a_1-\lceil a_2/a_1\rceil+1)\rceil-1=\lceil a_3/a_2(a_2/a_3-t)\rceil-1\leq 1-1=0.$$ Hence $\displaystyle f_1=x_1^{a_1-1}x_2^{{\lceil a_2/a_1\rceil}-1}$. Now by Proposition \ref{monomialofvint}, we have $(\overline I:f_1)=P$ and by Theorem  \ref{gen} part $(1)$, we get $\displaystyle\v(\overline{I^n})\geq a_1+\lceil\frac{a_2}{a_1}\rceil-2=a_1+\lceil\frac{a_2}{a_1}-\frac{a_2}{a_3}\rceil-1$. Hence $\v(\overline I)=\deg f_1$. Therefore by part $(3)(ii)$, for all $n\geq 1$, we have $\displaystyle\v(\overline{I^n})=(n-1)a_1+\deg f_1=na_1+\lceil\frac{a_2}{a_1}\rceil-2.$
	 \\$(v)$ We first show that $\v(\overline I)=\deg f_1$. Then the result follows from part $(3)(ii)$. If $a_1$ divides $a_2$ then $\v(\overline I)=\deg f_1$ follows from part $(3)$ $(iii)$. Hence, we assume that $a_1$ does not divide $a_2$.  
	 \\Note that $\displaystyle\lceil\frac{(a_1-1)a_3}{a_1a_2}\rceil=1$ implies $\displaystyle\lceil\frac{la_3}{a_1a_2}\rceil=1$ for all $1\leq l\leq a_1-1$. 
	 Let $ma_2=s_m a_1+t_m$ for some $s_m\geq 1$ and $0\le t_m<a_1$. Now $0< t_1<a_1$ and thus $\displaystyle\frac{a_2}{a_1}-\lceil\frac{a_2}{a_1}\rceil+1=\frac{t_1}{a_1}-\lceil\frac{t_1}{a_1}\rceil+1=\frac{t_1}{a_1}$. Therefore  \begin{eqnarray*}
	 \displaystyle	\deg f_1&=&a_1+\lceil\frac{a_2}{a_1}\rceil+\lceil\frac{a_3}{a_2}(\frac{a_2}{a_1}-\lceil\frac{a_2}{a_1}\rceil+1)\rceil-3=a_1+\lceil\frac{a_2}{a_1}\rceil+\lceil\frac{t_1a_3}{a_1a_2}\rceil-3\\&=& a_1+\lceil\frac{a_2}{a_1}\rceil+1-3=a_1+\lceil\frac{a_2}{a_1}\rceil-2.
	 \end{eqnarray*}
	 For all $1\leq m\leq a_1-1$, using Lemma \ref{ceil} ( taking $L=m$, $A=a_2/a_1$, $t=1$ and $c=s=0$), we have,
	 \begin{eqnarray*}
	 	\displaystyle\deg f_m&=& a_1-m+\lceil\frac{ma_2}{a_1}\rceil+\lceil\frac{a_3}{a_2}(\frac{ma_2}{a_1}-\lceil\frac{ma_2}{a_1}\rceil+1)\rceil-2\\ &=&  a_1-m+\lceil\frac{ma_2}{a_1}\rceil+\lceil\frac{a_3}{a_2}(s_m+\frac{t_m}{a_1}-s_m-\lceil\frac{t_m}{a_1}\rceil+1)\rceil-2 \\ &\geq&  a_1-m+m-1+\lceil\frac{a_2}{a_1}\rceil+\lceil\frac{a_3}{a_2}(\frac{t_m}{a_1}-\lceil\frac{t_m}{a_1}\rceil+1)\rceil-2\\ &\geq&  a_1+\lceil\frac{a_2}{a_1}\rceil+\lceil\frac{a_3}{a_2}(\frac{t_m}{a_1}-\lceil\frac{t_m}{a_1}\rceil+1)\rceil-3. 
	 \end{eqnarray*}
	 If $t_m=0$ then $$\displaystyle\deg f_m\geq  a_1+\lceil\frac{a_2}{a_1}\rceil+\lceil\frac{a_3}{a_2}\rceil-3\geq a_1+\lceil\frac{a_2}{a_1}\rceil+\lceil\frac{a_3}{a_1a_2}\rceil-3=a_1+\lceil\frac{a_2}{a_1}\rceil-2\geq\deg f_1.$$ If $1\leq t_m\leq a_1-1$ then $\displaystyle\lceil\frac{t_m}{a_1}\rceil=1$ and $\displaystyle\lceil\frac{t_ma_3}{a_1a_2}\rceil=1$. Hence $$\displaystyle\deg f_m\geq a_1+\lceil\frac{a_2}{a_1}\rceil+\lceil\frac{t_ma_3}{a_1a_2}\rceil-3=a_1+\lceil\frac{a_2}{a_1}\rceil-2=\deg f_1.$$
	 \end{proof}
	 \begin{Example}{\em
	 	Consider the ideal $I=\langle x_{1}^{4},x_2^{7},x_3^{77}\rangle$. Then $\lceil 7/4-7/77\rceil=\lceil 7/4\rceil$ and $\v(\overline I)=\deg f_3$. For the ideal $I=\langle x_{1}^{5},x_2^{8},x_3^{100}\rangle$, we have $\lceil 8/5-8/100\rceil=\lceil 8/5\rceil$ and $\v(\overline I)=\deg f_2$.
	 }
	 	\end{Example}
In the final part of our study, we analyze the $\v$-numbers of integral closure filtartions of complete intersection ideals. As a starting point, we prove the following two foundational results. The following result provides lower bounds for local $\v$-numbers and hence a lower bound for the $\v$-number of a height two complete intersection ideal.
\begin{Proposition}{\label{ht2}}
	 Let $I=\langle u_1,u_2\rangle$ be a complete intersection ideal in $S$. Let $u_1=x_{i_1}^{\alpha_1}\cdots x_{i_r}^{\alpha_r}$ and $u_2=x_{j_1}^{\beta_1}\cdots x_{j_l}^{\beta_l}$  with $\alpha_{s},\beta_{q}\geq 1$ for all $1\le s\le r$ and $1\le q\le l$. Let $P=\langle x_{i_p},x_{j_t}\rangle \in\Ass(I)$ and  $f=\prod\limits_{s=1}^{r}x_{i_s}^{b_s}\prod\limits_{q=1}^{l}x_{j_q}^{c_q}$ be a monomial such that $(\overline{I^n}:f)=P$. Then the following hold.
	\begin{enumerate}
		\item[$(i)$] $b_p=n\alpha_p-m$ for some $1\leq m\leq n\alpha_p$ and $\displaystyle c_t=\lceil\frac{m\beta_t}{\alpha_p}\rceil-1$.
		\item[$(ii)$] $\displaystyle c_j\geq\frac{m\beta_j}{\alpha_p}$ for all $1\leq j\leq l $ and $j\neq t$.
		\item[$(iii)$] $\displaystyle b_i> n\alpha_i-\frac{m\alpha_i}{\alpha_p}  $ for all $1\leq i\leq r$ and $i\neq p$. In particular, if $\alpha_p=\max\{\alpha_j:1\leq j\leq r\}$ and $x_{i_p}\in\mathcal G(P)$ then $\displaystyle b_i\geq  n\alpha_i-m+1$ for all $i\neq p$.
	\end{enumerate}
	\end{Proposition}
	\begin{proof}
		Let  $I=\bigcap\limits_{i=1}^h Q_i$ be a minimal irreducible decomposition of $I$. Then 
		$I^{(n)}=I^n=\bigcap\limits_{i=1}^hQ_i^n$ for all $n\geq 1$. Therefore by \cite[Theorem 4.1]{GMBV}, ${\overline{I^n}}=\bigcap\limits_{i=1}^h{\overline{Q_i^n}}$ for all $n\geq 1$. We denote $x_{i_s}$ by $x_s$ and $x_{j_q}$ by $y_q$ for all $1\leq s\leq r$ and $1\leq q\leq l$. Thus $P=\langle x_p,y_t\rangle$ and $f=\prod\limits_{i=1}^{r}x_i^{b_i}\prod\limits_{j=1}^{l}y_j^{c_j}$. The $P$-primary component of ${\overline{I^n}}$ is $Q=\langle x_p^{\alpha_p},y_t^{\beta_t}\rangle$.
		\\Now $P=(\overline{I^n}:f)=\bigcap\limits_{i=1}^h(\overline{Q_i^n}:f)\subseteq(\overline{Q_i^n}:f)$. If $f\notin \overline{Q_i^n}$, then $P\subset (\overline{Q_i^n}:f)\subseteq \sqrt{Q_i}$ for all $Q_i$. Since $\overline{Q_i^n}$ is $\sqrt{Q_i}$-primary  and $P\nsubseteq \sqrt{Q_i}$ for all $Q_i\neq Q$, we have $f\in\overline{Q_i^n}$ for all $Q_i\neq Q$ and $(\overline{Q^n}:f)=(\overline{I^n}:f)=P$. 
		\\  $(i)$ By Remark \ref{powers}, $\overline{Q^n}=\overline{\langle{x_p^{n\alpha_p},y_t^{n\beta_t}}\rangle}$ for all $n\geq 1$. Since $f\notin \overline{Q^n}$,  $b_p=n\alpha_p-m$ for some $1\leq m\leq n\alpha_p$ and by Lemma \ref{membership}, we get $\displaystyle\frac{n\alpha_p-m}{n\alpha_p}+\frac{c_t}{n\beta_t}<1$. Hence $\displaystyle c_t<\frac{m\beta_t}{\alpha_p}$. Since $fy_{t}\in \overline{Q^n}$, again using Lemma \ref{membership}, we get  $\displaystyle\frac{n\alpha_p-m}{n\alpha_p}+\frac{c_t+1}{n\beta_t}\geq 1$. Then $\displaystyle c_t+1\geq \frac{m\beta _t}{\alpha_p}$. Now $c_t\in\mathbb{Z}_{>0}$ and $\displaystyle c_t+1\geq \frac{m\beta_t}{\alpha_r}>c_t$ implies $\displaystyle c_t+1=\lceil\frac{m\beta_t}{\alpha_p}\rceil$ and hence $\displaystyle c_t=\lceil\frac{m\beta_t}{\alpha_p}\rceil-1$.
		\\$(ii)$ Fix $j\neq t$ and let $Q'=\langle x_p^{\alpha_p},y_j^{\beta_j}\rangle$. Then $f\in \overline{Q'^n}$. By Remark \ref{powers}, $\overline{Q'^n}=\overline{\langle{x_p^{n\alpha_p},y_j^{n\beta_j}}\rangle}$ for all $n\geq 1$. Hence, by Lemma \ref{membership}, we have $\displaystyle\frac{n\alpha_p-m}{n\alpha_p}+\frac{c_j}{n\beta_j}\geq 1 $ and  $\displaystyle c_j\geq \frac{m\beta_j}{\alpha_p}$. 
		\\$(iii)$ Fix $i\neq p$ and let $Q'=\langle x_i^{\alpha_i},y_t^{\beta_t}\rangle$. Then $f\in\overline{Q'^n}$. By Remark \ref{powers}, $\overline{Q'^n}=\overline{\langle{x_i^{n\alpha_i},y_t^{n\beta_t}}\rangle}$ for all $n\geq 1$. Therefore, by Lemma \ref{membership}, we get $\displaystyle\frac{b_i}{n\alpha_i}+\frac{1}{n\beta_t}(\lceil\frac{m\beta_t}{\alpha_p}\rceil-1)\geq 1$ and $$\displaystyle\frac{b_i}{n\alpha_i}\geq 1-\frac{1}{n\beta_t}(\lceil\frac{m\beta_t}{\alpha_p}\rceil-1)>1-\frac{1}{n\beta_t}(\frac{m\beta_t}{\alpha_p})=1-\frac{m}{n\alpha_p}.$$ Thus $\displaystyle b_i>n\alpha_i-\frac{m\alpha_i}{\alpha_p}$. If $\alpha_p=\max\{\alpha_j:1\leq j\leq r\}$ then 
		$b_i> n\alpha_i-m$. Hence $b_i\geq n\alpha_i-m+1$ for all $i\neq p$.
		\end{proof}	
	\begin{Proposition}{\label{ICCI}}
		Let $I=\langle u_1,\dots,u_r\rangle$ be a complete intersection monomial ideal in $S$ with $u_i=x_{i_1}^{\alpha_{i_1}}\cdots x_{i_{t_i}}^{\alpha_{i_{t_i}}}$ with $\alpha_{i_j}\geq 1$ for all $1\le i\le r$ and $1\le j\le t_i$.  If for any $j\in\{1,\ldots,r\}$, $|\supp(u_j)|=1$ then consider $\prod\limits_{p=2}^{t_j}x_{j_p}^{\alpha_{j_p}}=1$. Consider the primary component $Q=\langle x_{1_1}^{\alpha_{1_1}},\dots,x_{r_1}^{\alpha_{r_1}}\rangle$  of $I$ with $P=\sqrt{Q}$.  Let $$\displaystyle f=x_{1_1}^{n\alpha_{1_1}-1}x_{l_1}^{\lceil\frac{\alpha_{l_1}}{\alpha_{1_1}}\rceil-1}\prod\limits_{j=2}^{t_1}x_{1_j}^{n\alpha_{1_j}}\prod\limits_{i\in A}\prod\limits_{j=2}^{t_i} x_{i_j}^{\lceil\frac{\alpha_{i_j}}{\alpha_{1_1}}\rceil}$$ where $\alpha_{l_1}=\max\{\alpha_{i_1}\mid 1\le i\le r \}$ and $A=\{i\mid 2\le i\le r\text{ and } |\supp(u_i)| \geq2\}$.
		Then $(\overline{I^n}:f)=P$.
		\end{Proposition}	
		\begin{proof}
			Let  $I=\bigcap\limits_{i=1}^s Q_i$ be a minimal irreducible decomposition of $I$.  Then 
	$I^{(n)}=I^n=\bigcap\limits_{i=1}^sQ_i^n$ for all $n\geq 1$ and by \cite[Theorem 4.1]{GMBV}, ${\overline{I^n}}=\bigcap\limits_{i=1}^s{\overline{Q_i^n}}$ for all $n\geq 1$. We first show that if $Q_i\neq Q$ then $f\in\overline{Q_i^n}$. 
	\\
	\textbf{Case 1}: Suppose $x_{1_1}^{\alpha_{1_1}}\notin Q_i$. Then $|\supp(u_1)|\geq 2$ and there exists $x_{1_j}$ for some $2\le j\le t_1$ such that $x_{1_j}^{\alpha_{1_j}}\in Q_i$. Since by Remark \ref{powers}, $x_{1_j}^{n\alpha_{1_j}}\in\mathcal G(\overline{Q_i^n})$ and $x_{1_j}^{n\alpha_{1_j}}\mid f$, we have $f\in\overline{Q_i^n}$. 
	\\ \textbf{Case 2}: Suppose $x_{1_1}^{\alpha_{1_1}}\in Q_i$. Since $Q_i\neq Q$, we have $A\neq\{\emptyset\}$ and there exists $x_{p_j}^{\alpha_{p_j}}\in Q_i\setminus Q$ for some $p\in A$ and $2\le j\le t_p$. By Remark \ref{powers}, $x_{1_1}^{n\alpha_{1_1}}, x_{p_j}^{n\alpha_{p_j}}\in\mathcal G( \overline{Q_i})$.
	 Note that,  $$\displaystyle\frac{n\alpha_{1_1}-1}{n\alpha_{1_1}}+\frac{1}{n\alpha_{p_j}}(\lceil\frac{\alpha_{p_j}}{\alpha_{1_1}}\rceil)\geq 1-\frac{1}{n\alpha_{1_1}}+\frac{1}{n\alpha_{1_1}}=1.$$ Therefore by Lemma \ref{membership}, $\displaystyle x_{1_1}^{n\alpha_{1_1}-1}x_{p_j}^{\lceil\frac{\alpha_{p_j}}{\alpha_{1_1}}\rceil}\in\overline{Q_i^n}$. Since $\displaystyle x_{1_1}^{n\alpha_{1_1}-1}x_{p_j}^{\lceil\frac{\alpha_{p_j}}{\alpha_{1_1}}\rceil}$ divides $f$, we have $f\in\overline{Q_i^n}$. 
	\\Now we show that $(\overline{Q^n}:f)=P$. Let $\displaystyle f_1=x_{1_1}^{n\alpha_{1_1}-1}x_{l_1}^{\lceil\frac{\alpha_{l_1}}{\alpha_{1_1}}\rceil-1}$. Since $x_{i_j}\notin P$ for all $1\leq i\leq r$ and $2\le j\le t_i$, it suffices to show that $(\overline{Q^n}:f_1)=P$. By Remark \ref{powers}, $\overline{Q^n}=\overline{\langle x_{1_1}^{n\alpha_{1_1}},\dots,x_{r_1}^{n\alpha_{r_1}}\rangle}$. Since $$\displaystyle\frac{n\alpha_{1_1}-1}{n\alpha_{1_1}}+\frac{1}{n\alpha_{l_1}}(\lceil\frac{\alpha_{l_1}}{\alpha_{1_1}}\rceil-1)<1-\frac{1}{n\alpha_{1_1}}+\frac{1}{n\alpha_{l_1}}\frac{\alpha_{l_1}}{\alpha_{1_1}}=1,$$ by Lemma \ref{membership}, we get $f_1\notin\overline{Q^n}$. Thus $\Ass(R/(\overline{Q^n}:f_1))=\Ass(R/\overline{Q^n})=\{P\}$. Since $\alpha_{l_1}\geq\alpha_{i_1}$ for all $1\le i\le r$, we have
	\begin{eqnarray*}
	\displaystyle\frac{n\alpha_{1_1}-1}{n\alpha_{1_1}}+\frac{1}{n\alpha_{l_1}}(\lceil\frac{\alpha_{l_1}}{\alpha_{1_1}}\rceil-1)+\frac{1}{n\alpha_{i_1}} &\geq&\displaystyle\frac{n\alpha_{1_1}-1}{n\alpha_{1_1}}+\frac{1}{n\alpha_{l_1}}(\frac{\alpha_{l_1}}{\alpha_{1_1}}-1)+\frac{1}{n\alpha_{i_1}}\\&\geq& 1-\frac{1}{n\alpha_{1_1}}+\frac{1}{n\alpha_{1_1}}-\frac{1}{n\alpha_{l_1}}+\frac{1}{n\alpha_{i_1}}\geq 1.	\end{eqnarray*}
	Hence, by Lemma \ref{membership}, we have $f_1x_{i_1}\in\overline{Q^n}$ for all $1\le i\le r$ and $P\subset (\overline{Q^n}:f_1)$. Therefore $(\overline{I^n}:f)=(\overline{Q^n}:f)\bigcap (\bigcap\limits_{Q_i\neq Q}(\overline{Q_i^n}:f))=P$. 
		\end{proof}	
	We need the following technical lemma to prove the next result.
\begin{Lemma}\label{b>bmoda}
	Let $a,b\in\mathbb{Z}_>0$. Then $\displaystyle b\geq\lceil\frac{b}{a}\rceil+1$ if and only if $b\geq 2$ and $a\geq 2$.
\end{Lemma}
\begin{proof} Suppose $\displaystyle b\geq\lceil\frac{b}{a}\rceil+1$. If $a=1$ then $b=\displaystyle\lceil\frac{b}{a}\rceil$ which is a contradiction. Thus $a\geq 2$. Since $b>0$, we have $\displaystyle\lceil\frac{b}{a}\rceil\geq 1$ and hence $b\geq 2$. Conversely, suppose $a,b\geq 2$. Since $\displaystyle b-\frac{b}{a}\geq b-\frac{b}{2}=\frac{b}{2}\geq 1$, we have $\displaystyle b\geq \lceil\frac{b}{a}\rceil+1$.
\end{proof}
With the preceding developments, we are now ready to prove the final main result of our study.
	\begin{Theorem}\label{uppbndcompleteint}
	Let $I=\langle u_1,\dots,u_r\rangle$ be a complete intersection monomial ideal in $S$ with $\deg u_1=\alpha(I)$. Then the following hold.
	\begin{itemize} 
		\item[$(1)$] 
		\begin{enumerate}
			\item[$(i)$] $\v(\overline{I^n})\leq\v(I^n)$ for all $n\geq 1$.
			\item[$(ii)$] If $u_1$ and $u_t$ are not squarefree for some $2\le t\le r$, then $\v(\overline{I^n})\le\v(I^n)-1$ for all $n\geq 1$.
			\item[$(iii)$] Let $I$ be an equigenerated complete intersection ideal with $u_i=m_i^\alpha$  where $m_i$ are squarefree monomials for all $1\leq i \leq r$. Then for all $n\geq 1$, $$\reg(S/\overline{I^n})=\v(\overline{I^n})=n\alpha(I)+(|\supp(u_1)|-1)(r-1)-1.$$
		\end{enumerate}	
		\item[$(2)$] Suppose $r=2$, $u_1=(x_{i_1}\cdots x_{i_q})^\alpha$ and $u_2=x_{j_1}^{\beta_1}\cdots x_{j_l}^{\beta_l}$ with $1\leq \alpha\leq\beta_1\leq\cdots\leq\beta_l$ and $q\leq l$. Then  for all $n\geq 1$ and $P\in \Ass(\overline{I^n})$,	\vspace{-2mm} $$\displaystyle\v(\overline{I^n})=\v_P(\overline{I^n})=n\alpha(I)+\sum\limits_{j=1}^l\lceil\frac{\beta_j}{\alpha}\rceil-2.	\vspace{-2mm}$$
	 In particular, if $\alpha=1$ then  $\v(I^n)=\v(\overline{I^n})$ for all $n\geq 1$.
		\vspace{2mm}
		\item[$(3)$] Suppose $r=2$, $u_1=x_{i_1}^\alpha$, $u_2=x_{j_1}^{\beta_1}\cdots x_{j_l}^{\beta_l}$ with $l\geq 2$ and $\deg u_2=\alpha$. Then $\v(\overline{I^n})\geq n\alpha(I)$ for all $n\geq 1$. 
		\\Moreover, if there exists $j$ with $2\beta_j\geq \alpha$ then $\v(\overline{I^n})=n\alpha(I)$ for all $n\geq 1$.\vspace{2mm}
		\item[$(4)$] For any integer $a\geq 1$, there exists a height two equigenerated complete intersection monomial ideal $I$  such that $\reg(S/\overline{I^n})-\v(\overline{I^n})=a-1$ for all $n\geq 1$.
	\end{itemize}
\end{Theorem}
\begin{proof}
	$(1)$ $(i)$ By Proposition \ref{ICCI}, we get, $\v(\overline{I^n})\le\v_P(\overline{I^n})\le \deg f$ where $f$ is the monomial in Proposition \ref{ICCI} and $P\in\Ass(\overline{I^n})$. Thus, by Theorem \ref{ci},  for all $n\geq 1$, we have \vspace{-2mm} \begin{eqnarray*}
	\displaystyle \v(I^n)-\deg f&=&n\alpha(I)+\sum\limits_{i=2}^r \deg u_i-r -(n\alpha(I)-1+\lceil\frac{\alpha_{l_1}}{\alpha_{1_1}}\rceil-1+\sum\limits_{i\in A}\sum\limits_{j=2}^{t_i}\lceil\frac{\alpha_{i_j}}{\alpha_{1_1}}\rceil)\\ &=& \sum\limits_{i=2}^r\sum\limits_{j=1}^{t_i}\alpha_{i_j}-r+1-\sum\limits_{i\in A}\sum\limits_{j=2}^{t_i}\lceil\frac{\alpha_{i_j}}{\alpha_{1_1}}\rceil-\lceil\frac{\alpha_{l_1}}{\alpha_{1_1}}\rceil+1\\ &=&\sum\limits_{i=2}^r\alpha_{i_1}+\sum\limits_{i\in A}\sum\limits_{j=2}^{t_i}(\alpha_{i_j}-\lceil\frac{\alpha_{i_j}}{\alpha_{1_1}}\rceil)-(r-1)-\lceil\frac{\alpha_{l_1}}{\alpha_{1_1}}\rceil+1\\ &\geq& \sum\limits_{i=2}^r \alpha_{i_1}-(r-1)-\lceil\frac{\alpha_{l_1}}{\alpha_{1_1}}\rceil+1. 
\end{eqnarray*}
If $\alpha_{l_1}=\alpha_{1_1}$, then $1-\displaystyle\lceil\frac{\alpha_{l_1}}{\alpha_{1_1}}\rceil=0$. Since $\sum\limits_{i=2}^r\alpha_{i_1}-(r-1)\geq 0$, we have $\v(I^n)\geq \deg f$. If $\alpha_{l_1}>\alpha_{1_1}$, then 
$\displaystyle\v(I^n)-\deg f\geq \sum\limits_{\substack{i=2\\ i\ne l}}^{r}\alpha_{i_1}-(r-2)+\alpha_{l_1}-\lceil\frac{\alpha_{l_1}}{\alpha_{1_1}}\rceil\geq 0.$
Thus $\v(\overline{I^n})\le\v(I^n)$ for all $n\geq 1$. 
\\$(ii)$ Suppose $u_1$ and $u_t$ are not squarefree for some $2\le t\le r$. Without loss of generality, we assume that $\alpha_{1_1}\geq 2$ and $\alpha_{t_1}\geq 2$. 
Therefore, for all $n\geq 1$, we have $$\v({I^n})-\v(\overline{I^n})\geq \v({I^n})-\deg f\geq \displaystyle\sum\limits_{i=2}^{r}\alpha_{i_1}-(r-2)-\lceil\frac{\alpha_{l_1}}{\alpha_{1_1}}\rceil.$$
  If $\alpha_{l_1}>\alpha_{1_1}\geq 2$, then by Lemma \ref{b>bmoda}, $\displaystyle\alpha_{l_1}-\lceil\frac{\alpha_{l_1}}{\alpha_{1_1}}\rceil\geq 1$. Thus, for all $n\geq 1$, we have $$\v(I^n)-\deg f\geq \sum\limits_{i=2}^r \alpha_{i_1}-(r-2)-\lceil\frac{\alpha_{l_1}}{\alpha_{1_1}}\rceil=\sum\limits_{\substack{i=2\\ i\ne l}}^{r}\alpha_{i_1}-(r-2)+\alpha_{l_1}-\lceil\frac{\alpha_{l_1}}{\alpha_{1_1}}\rceil\geq1.$$ 
Suppose $\alpha_{l_1}=\alpha_{1_1}\geq 2$. Then for all $n\geq 1$, we have  \vspace{-2mm}
\begin{eqnarray*}\v(I^n)-\deg f\geq \sum\limits_{i=2}^r \alpha_{i_1}-(r-2)-\lceil\frac{\alpha_{l_1}}{\alpha_{1_1}}\rceil&=&\sum\limits_{\substack{i=2\\ i\ne t}}^{r}\alpha_{i_1}-(r-2)+\alpha_{t_1}-\lceil\frac{\alpha_{l_1}}{\alpha_{1_1}}\rceil \\ &\geq& (r-2)-(r-2)+2-1=1.\end{eqnarray*} 
\\$(iii)$ Consider the squarefree monomial ideal $J=\langle m_1,\ldots,m_r\rangle$. Note that $I$ is a complete intersection ideal implies $J$ is a  complete intersection ideal. Then by \cite[Theorem 1.4.6]{HH} and Remark \ref{powers}, for all $n\geq 1$, we have $$J^{n\alpha}=\overline{J^{n\alpha}}=\overline{\langle m_1,\ldots,m_r\rangle^{n\alpha}}=\overline{\langle m_1^{n\alpha},\ldots,m_r^{n\alpha}\rangle}=\overline{\langle m_1^\alpha,\ldots,m_r^\alpha\rangle^{n}}=\overline{I^n}.$$ Since the monomials $m_i$ are squarefree and $I$ is equigenerated, we have $\deg m_i=\deg m_j$ for all $1\leq i,j\leq r$. This implies $|\supp(u_i)|=|\supp(m_i)|=\deg m_i=\deg m_j=|\supp(m_j)|=|\supp(u_j)|$ for all $1\leq i,j\leq r$. Hence, by \cite[Lemma 4.4]{BHT}, for all $n\geq 1$, we have $$\reg(J^{n\alpha})=n\alpha\deg m_1+(\deg m_1-1)(r-1)=n\alpha(I)+(|\supp(u_1)|-1)(r-1).$$ Now using Theorem \ref{ci}, for all $n\geq 1$, we have 
\begin{eqnarray*}\v(\overline{I^n})&=&\v(J^{n\alpha})=n\alpha\deg m_1+(r-1)\deg m_1-r\\&=&n\alpha(I)+(|\supp(u_1)|-1)(r-1)-1=\reg(S/\overline{I^n}).\end{eqnarray*}
$(2)$ We denote $x_{i_p}$ by $x_p$ and $x_{j_t}$ by $y_t$ for all $1\leq p\leq q$ and $1\leq t\leq l$. Note that  $I=\bigcap\limits_{\substack{1\leq i\leq q\\1\leq j\leq l}} Q_{ij}$ is a minimal irreducible decomposition of $I$ where $Q_{ij}=\langle x_i^\alpha,y_j^{\beta_j}\rangle$. By \cite[Theorem 4.1]{GMBV}, ${\overline{I^n}}=\bigcap\limits_{\substack{1\leq i\leq q\\1\leq j\leq l}}{\overline{Q_{ij}^n}}$ for all $n\geq 1$ and and $\Ass({\overline{I^n}})=\{P_{ij}=\sqrt{Q_{ij}}: 1\leq i\leq q, 1\leq j\leq l\}$ for all $n\geq 1$.  By Remark \ref{powers}, $\overline{Q_{ij}^n}=\overline{\langle x_i^{n\alpha},y_j^{n\beta_j}\rangle}$. We show that $\displaystyle\v_{P_{ij}}(\overline{I^n})=qn\alpha+\sum\limits_{h=1}^l\lceil\frac{\beta_h}{\alpha}\rceil-2$ for all $1\leq i\leq q$ and $1\leq j\leq l$. Note that by Proposition \ref{ICCI}, we have 
$\displaystyle\v_{P_{ij}}(\overline{I^n})\leq qn\alpha+\sum\limits_{h=1}^l\lceil\frac{\beta_h}{\alpha}\rceil-2$.
\\Let $g=\prod\limits_{p=1}^{q}x_{p}^{b_p}\prod\limits_{h=1}^{l}y_{h}^{c_h}$ be a monomial such that $(\overline{I^n}:g)=P_{ij}$. Then by Propossition \ref{ht2}, $b_i=n\alpha-m$ for some $1\leq m\leq n\alpha$, $\displaystyle c_j=\displaystyle\lceil\frac{m\beta_j}{\alpha}\rceil-1$, $\displaystyle c_h\geq\lceil\frac{m\beta_h}{\alpha}\rceil$ for all $1\leq h\leq l $ and $h\neq j$, and $b_p\geq n\alpha-m+1$ for all $1\leq p\leq q $ and $p\neq i$. Therefore 
 \begin{eqnarray*}
 	\deg g&\geq& (n\alpha-m)+\sum\limits_{\substack{p=1\\ p\neq i}}^q (n\alpha-m+1)+\displaystyle\lceil\frac{m\beta_j}{\alpha}\rceil-1+\sum\limits_{\substack{h=1\\ h\neq j}}^l \frac{m\beta_h}{\alpha}\rceil\\&=& qn\alpha-qm+(q-1)+\sum\limits_{h=1}^l\lceil\frac{m\beta _h}{\alpha}\rceil-1\\&= &qn\alpha-qm+(q-2)+\sum\limits_{h=1}^l\lceil\frac{m\beta_h}{\alpha}\rceil\\&\geq & qn\alpha-q(m-1)-2+l(m-1)+\sum\limits_{h=1}^l\lceil\frac{\beta_h}{\alpha}\rceil\\&= &qn\alpha+\sum\limits_{h=1}^l\lceil\frac{\beta_h}{\alpha}\rceil-2+(l-q)(m-1)\geq qn\alpha+\sum\limits_{h=1}^l\lceil\frac{\beta_h}{\alpha}\rceil-2
 \end{eqnarray*} \vspace{-2mm} where the fourth inequality follows from Lemma \ref{ceil} ( taking $L=m$, $A=\beta_h/\alpha$, $t=1$ and $c=s=0$). Hence $\displaystyle\v_{P_{ij}}(\overline{I^n})=n\alpha(I)+\sum\limits_{h=1}^l\lceil\frac{\beta_h}{\alpha}\rceil-2.$
 \\The last part follows from the Theorem \ref{ci}.
\\$(3)$ We denote $x_{i_1}$ by $x$ and $x_{j_q}$ by $y_q$ for all $1\leq q\leq l$. Then by Remark \ref{powers}, $\overline{I^n}={\overline{\langle x^{n\alpha}, y_1^{n\beta_1}\cdots y_l^{n\beta_l}\rangle}}$ for all $n\geq 1$. Note that $I=\bigcap\limits_{i=1}^l Q_i$ is a minimal irreducible decomposition of $I$ where $Q_i=\langle x^\alpha,y_i^{\beta_i}\rangle$. By \cite[Theorem 4.1]{GMBV}, ${\overline{I^n}}=\bigcap\limits_{i=1}^l{\overline{Q_i^n}}$ for all $n\geq 1$ and $\Ass({\overline{I^n}})=\{P_i=\sqrt{Q_i}: 1\leq i\leq l\}$ for all $n\geq 1$. By Remark \ref{powers}, $\overline{Q_i^n}=\overline{\langle x^{n\alpha},y_i^{n\beta_i}\rangle}$ for all $n\geq 1$ and $1\leq i\leq l$. We show that $\v_{P_i}({\overline{I^n}})\geq n\alpha$ for all $n\geq 1$ and $1\leq i\leq l$. 
\\Let $f=x^c\prod\limits_{j=1}^{l}y_j^{c_j}$ and $(\overline{I^n}:f)=P_i$. Then by Proposition \ref{ht2}, we have $c=n\alpha-m$ for some $1\leq m\leq n\alpha$, $\displaystyle c_i=\lceil\frac{m\beta_i}{\alpha}\rceil-1$ and $\displaystyle c_j\geq \frac{m\beta_j}{\alpha}$  for all $j\neq i$. Hence \begin{eqnarray*}\displaystyle c+\sum\limits_{j=1}^l c_j\geq n\alpha-m+\sum\limits_{\substack{j=1\\ j\ne i}}^{l}\frac{m\beta_j}{\alpha}+\lceil\frac{m\beta_i}{\alpha}\rceil-1&=&n\alpha-m+\frac{m(\alpha-\beta_i)}{\alpha}+\lceil\frac{m\beta_i}{\alpha}\rceil-1\\&=&n\alpha-\frac{m\beta_i}{\alpha}+\lceil\frac{m\beta_i}{\alpha}\rceil-1.\end{eqnarray*} 
Now we show that $\alpha$ does not divide $m\beta_i$. Suppose $\alpha$ divides $m\beta_i$. As $fx\in \overline{I^n}\subset \overline{Q_i^n}$, using 
Lemma \ref{membership}, we have $$\displaystyle1\leq \frac{c+1}{n\alpha}+\frac{c_i}{n\beta_i}=\frac{n\alpha-m+1}{n\alpha}+\frac{\frac{m\beta_i}{\alpha}-1}{n\beta_i}=1+\frac{1}{n\alpha}-\frac{1}{n\beta_i}<1$$ which is a contradiction (where the last inequality follows as $l\geq 2$ and hence $\sum\limits_{j=1}^l\beta_j=\alpha$ implies $\beta_i<\alpha$). Therefore $\displaystyle\lceil\lceil\frac{m\beta_i}{\alpha}\rceil-\frac{m\beta_i}{\alpha}-1 \rceil=0$ and hence $$\displaystyle c+\sum\limits_{j=1}^l c_j\geq n\alpha+\lceil\lceil\frac{m\beta_i}{\alpha}\rceil-\frac{m\beta_i}{\alpha}-1 \rceil= n\alpha.$$ This implies $\v_{P_i}(\overline{I^n})\geq n\alpha$ for all $1\leq i\leq r$ and $n\geq 1$. 
\\Suppose there exists $j\in\{1,\ldots,r\}$ with $2\beta_j\geq \alpha$. Consider the monomial $\displaystyle g=xy_j^{n\beta_j-1}\prod\limits_{\substack{i=1\\i\neq j}}^{l}y_i^{n\beta_i}$. Note that $g\in \overline{Q_i^n}=\overline{\langle x^{n\alpha},y_i^{n\beta_i}\rangle}$ for all $i\neq j$. Since $\displaystyle  \frac{1}{n\alpha}+\frac{n\beta_j-1}{n\beta_j}=1-\frac{\alpha-\beta_j}{n\alpha\beta_j}<1$ (as $l\geq 2$ and hence $\alpha>\beta_j$), by Lemma \ref{membership}, we have $g\notin \overline{Q_j^n}=\overline{\langle x^{n\alpha},y_j^{n\beta_j}\rangle}$. Note that $gy_j\in \overline{Q_j^n}$. Since $\displaystyle  \frac{2}{n\alpha}+\frac{n\beta_j-1}{n\beta_j}=1+\frac{2\beta_j-\alpha}{n\alpha\beta_j}\geq 1,$ again using Lemma \ref{membership}, we have $gx\in \overline{Q_j^n}$. Hence $P_j\subseteq (\overline{Q_j^n}:g)\neq R$ and $\Ass(R/(\overline{Q_j^n}:g))=\Ass(R/\overline{Q_j^n})=\{P_j\}$. Thus $P_j= (\overline{Q_j^n}:f)$. Therefore $(\overline{I^n}:g)=(\overline{Q_j^n}:g)\bigcap (\bigcap\limits_{Q_i\neq Q_j}(\overline{Q_i^n}:g))=P_j$. Hence $\v_{P_j}(\overline{I^n})\leq \deg g=n\alpha$. Thus $\v(\overline{I^n})=n\alpha$ for all $n\geq 1$.
\\$(4)$ Let $\alpha$ be an integer such that $\alpha\geq 2a+1$. Consider the ideal $I=\langle x_1^\alpha, x_{a+2}^{\alpha-a}\prod\limits_{j=2}^{a+1}x_j \rangle$ in the polynomial ring $S=\mathbb C[x_1,\ldots,x_{a+2}]$. Then using part $(4)$, we have $\v(\overline{I^n})=n\alpha$ for all $n\geq 1$. By \cite[Example 3.2]{Hoa}, we have $\reg(S/\overline{I^n})=n\alpha+a-1$ for all $n\geq 1$ and thus we get the required result.  
\end{proof}
As a consequence, we obtain the following.
\begin{Corollary}\label{diff}
For any integer $q\geq 0$, there exist  an equigenerated complete inetersection monomial ideal $I$ and a non-equigenerated complete inetersection monomial ideal $J$ such that for all $n\geq 1$, $$\v(I^n)-\v(\overline{I^n})=q=\v(J^n)-\v(\overline{J^n})=q.$$ 
\end{Corollary}	
\begin{proof}
	Let $S=K[x_1,\ldots,x_m]$ be a polynomial ring over a field $K$ with $m\geq 2$. Consider the ideal $I=\langle x_i^{q+1},x_j^{q+1}\rangle$ with $i\neq j$. Then by Theorem \ref{ci} and part $(2)$ of Theorem \ref{gen}, we have $\v(I^n)-\v(\overline{I^n})=n(q+1)+(q+1)-2-(n(q+1)+1-2)=q$ for all $n\geq 1$.
	\\Now we construct the non-equigenerated ideal $J$. Suppose $q=0$. Then the result follows from part $(2)$ of Theorem \ref{uppbndcompleteint}. Let $q\geq 1$.
	Consider the polynomial ring $S=K[x_1,\ldots,x_q,y_1,\ldots,y_q]$ where $K$ is a field and the monomial ideal $J=\langle (x_1\cdots x_q)^2, (y_1\cdots y_q)^3\rangle\in S$. Then by Theorem \ref{ci} and part $(2)$ of Theorem \ref{uppbndcompleteint}, for all $n\geq 1$, we get $\v(J^n)-\v(\overline{J^n})=2nq+3q-2-(2nq+2q-2)=q$.
\end{proof}

	\end{document}